\documentclass[12pt,leqno]{article}
\usepackage{amsmath,amsthm,amssymb}
\allowdisplaybreaks

\newcommand{\ds}{\displaystyle}

\newcommand{\pa}{\partial}
\newcommand{\Om}{\Omega}
\newcommand{\om}{\omega}

\newcommand{\la}{\lambda}
\newcommand{\La}{\Lambda}
\newcommand{\va}{\varepsilon}
\newcommand{\be}{\beta}
\newcommand{\al}{\alpha}

\theoremstyle{plain}
\newtheorem{theorem}{Theorem}[section]

\newtheorem{lemma}[theorem]{Lemma}
\newtheorem{proposition}[theorem]{Proposition}
\newtheorem{corollary}[theorem]{Corollary}
\newif \ifLastSection \LastSectionfalse

\numberwithin{equation}{section}

\newcommand{\R}{{\mathbb R}}

\newcommand{\D}{D^{1,2}(\R^N)}
\newcommand{\N}{\R^N}
\newcommand{\s}{\frac{2(N-1)}{N-2}}

\newcommand{\U}{U_{\eta^1,\la_1}}
\newcommand{\V}{U_{\eta^2,\la_2}}
\newcommand{\Ui}{U_{\eta^i,\la_i}}
\newcommand{\Uj}{U_{\eta^j,\la_j}}
\newcommand{\bz}{\bar{z}}

\newcommand{\bx}{\bar{x}}
\title{ {\bf On the Webster Scalar Curvature Problem on the CR Sphere with a Cylindrical-type Symmetry}}
\author{Daomin Cao \\
\small Institute of Applied Mathematics, AMSS\\
 \small  Chinese Academy of Sciences\\
 \small Beijing, 100080, P. R. China\\
 \small  dmcao@amt.ac.cn
\and
Shuangjie Peng\\
 \small School of Mathematics and Statistics \\
 \small Central China Normal University\\
 \small Wuhan, 430079, P. R. China\\
 \small  sjpeng@mail.ccnu.edu.cn\\
 \and
Shusen Yan\\
 \small School of Mathematics, Statistics and Computer
Science,\\
\small The University of New England, Armidale, NSW 2351,
Australia\\
 \small syan@turing.une.edu.au}
\date{}
\begin{document}
\baselineskip=20pt

\maketitle
\begin{abstract} By  variational methods, for a kind of Webster scalar curvature
problems on the CR sphere with  cylindrically symmetric curvature,
we construct some multi-peak solutions  as the parameter is
sufficiently small under certain assumptions. We also obtain the
asymptotic behaviors of the solutions.
\end{abstract}



\section{Introduction and main results}
The Webster scalar curvature problem on the CR sphere can be
briefly discussed below. Let $\theta_0$ be the standard contact
form of the CR manifold $\mathbb{S}^{2n+1}$. Given a smooth
function $\bar{\Phi}$ on $\mathbb{S}^{2n+1}$, the Webster scalar
curvature problem on $\mathbb{S}^{2n+1}$ consists in finding a
contact form $\theta$ conformal to $\theta_0$ such that the
corresponding Webster scalar curvature is $\bar{\Phi}$. This
problem is equivalent to solve the following equation
\begin{equation}\label{W}
b_n\Delta_{\theta_0}v({\vartheta})+c_nv({\vartheta})=\bar{\Phi}(\vartheta)v(\vartheta)^{b_n-1},\,\,\,\,\vartheta\in
 \mathbb{S}^{2n+1},
\end{equation}
where $b_n=2+2/n$, $\Delta_{\theta_0}$ is the sub-Laplacian on $(
\mathbb{S}^{2n+1}, \theta_0)$ and $c_n=n(n+1)/2$ is the Webster
scalar curvature of $( \mathbb{S}^{2n+1}, \theta_0)$. If $v>0$
solves \eqref{W}, then $(\mathbb{S}^{2n+1}, v^{2/n}\theta_0)$ has
the Webster scalar curvature $\bar{\Phi}$. We refer to \cite{Je}
for a more detailed presentation for this problem.

Using the Heisenberg group $\mathbb{H}^n$ and the  CR equivalence
$F:\,\,\mathbb{S}^{2n+1}\setminus\{0,\cdot\cdot\cdot,-1\}\to
\mathbb{H}^n$,
\begin{equation}\label{H}
F(\vartheta_1,\cdot\cdot\cdot,\vartheta_{n+1})=\Bigl(\frac{\vartheta_1}{1+\vartheta_{n+1}},
\cdot\cdot\cdot,\frac{\vartheta_n}{1+\vartheta_{n+1}},
Re\bigl(i\frac{1-\vartheta_{n+1}}{1+\vartheta_{n+1}}\Bigl)\Bigl),
\end{equation}
Equation \eqref{W} becomes, up to an uninfluent constant,
\begin{equation}\label{P}
-\Delta_{\mathbb{H}^n}u(\zeta)=\Phi(\zeta)u(\zeta)^{\frac{Q+2}{Q-2}},\,\,\,\zeta\in
\mathbb{H}^n.
\end{equation}
Here $\Delta_{\mathbb{H}^n}$ is the Heisenberg sub-Laplacian,
$Q=2n+2$ is the homogeneous dimension of $\mathbb{H}^n$ and $\Phi$
corresponds to $\bar{\Phi}$ in the equivalence $F$.

In this paper, we shall give some existence results for the
concentration solutions to problem \eqref{W}  or \eqref{P}, under
suitable assumption on the prescribed  curvatures. In particular,
we shall mainly assume that the prescribed curvature $\Phi$ has a
natural symmetry, namely a cylindrical-type symmetry.

The Yamabe problem on CR manifolds has been extensively investigated
and  many interesting  results have been obtained, we can refer to
\cite{Ga1,Ga2,Gar, Je2,Je3}. On the contrary, concerning the Webster
scalar curvature problem, there are very few results established. In
recent years, there has been a growing interest on equations of the
same kind of \eqref{W} or \eqref{P} and various existence and
non-existence results inspired by this topic have been established
by several authors, for example, we can refer to
\cite{Bi1,Bi2,Ci,Gar,La,Ug} and the references therein. However,
these results are quite different in nature from the results we
shall prove in this paper and do not apply directly to the Webster
scalar curvature problem. Recently, in \cite{Ma}, Malchiodi and
Uguzzoni obtained an interesting result for problem \eqref{P} in the
perturbative case, i.e., when $\Phi$ is assumed to be a small
perturbation of a constant.

The aim of this paper is to study a natural case that problem
\eqref{P} has cylindrical curvatures $\Phi(Z,t)=\Phi(|Z|,t)$,
which correspond on $\mathbb{S}^{2n+1}$ to curvatures $\bar{\Phi}$
depending only on the last complex variable of
$\mathbb{S}^{2n+1}\subset \mathbb{C}^{n+1}$. Concerning this case,
in \cite{Fe} and \cite{Ca}, via the abstract Ambrosetti-Badiale
finite dimensional reduction method \cite{Am1,Am2}, some results
analogous to those found in \cite{Am3} in the Riemannian setting
were verified in the CR setting. However, we should point out
that, by  Corollary 1.3 below, the solutions found in \cite{Fe}
and \cite{Ca} are closed to the manifold
$\{V_{s,\la}(Z,t):\,\,\la>0,\,s\in\R\}$ (for the definition, see
\eqref{M}) and do not have the concentration properties. In the
present paper, we will construct some solutions which concentrate
on some maximum points of the prescribed curvature as some
parameter varies. In particular, our restriction on the prescribed
curvature is totally different from that in \cite{Fe} or
\cite{Ca}, more precisely, we only need to impose some kind of
flatness condition on each local maximum point of the prescribed
curvature.

To state our main results, we first give some notations.

Let us denote a point in
$\mathbb{H}^n=\mathbb{C}^n\times\R=\R^{2n}\times\R$ by
$\zeta=(Z,t),\,Z=x+iy$, and by
$$
\rho(\zeta)=(|Z|^4+t^2)^{\frac{1}{4}}
$$
the homogeneous norm in $\mathbb{H}^n$. In the sequel, we shall
always suppose that $\Phi$ is continuous and bounded on
$\mathbb{H}^n$, and $\Phi$ has cylindrical symmetry, i.e.
$\Phi(Z,t)=\Phi(|Z|,t)$. Define the space of cylindrically
symmetric functions of Folland-Stein Sobolev space
$S_0^1(\mathbb{H}^n)$, namely
$$
S^1_{cy}(\mathbb{H}^n)=\{u\in
S_0^1(\mathbb{H}^n):\,\,u(Z,t)=u(|Z|,t)\},
$$
where $S_0^1(\mathbb{H}^n)$ is defined as the completion of
$C_0^\infty(\mathbb{H}^n)$ with respect to the norm
$$
\|u\|^2_{S_0^1(\mathbb{H}^n)}=\int_{\mathbb{H}^n}|\nabla_{\mathbb{H}^n}u|^2dZdt.
$$
Let us observe that $S^1_{cy}(\mathbb{H}^n)$ is a Hilbert space
endowed with the scalar product $\langle u,v
\rangle=\int_{\mathbb{H}^n}\nabla_{\mathbb{H}^n}u\cdot
\nabla_{\mathbb{H}^n}v dV_{\mathbb{H}^n}$. It is known (see
\cite{Je3}) that all the positive cylindrical symmetric solutions in
$S_0^1(\mathbb{H}^n)$ to the problem
\begin{equation}\label{L1}
-\Delta_{\mathbb{H}^n}V=V^{\frac{Q+2}{Q-2}},\,\,\,V\in
S^1_0(\mathbb{H}^n)
\end{equation}
are of the form
\begin{equation}\label{M}
V_{s,\la}(Z,t)=c_0\la^{\frac{Q-2}{2}}V_0\Bigl(\la
|Z|,\la^2(t-s)\Bigl),
\end{equation}
where $\la>0$, $s\in\R$, $c_0$ is a suitable positive constant, and
$$
V_0(|Z|,t)=\Bigl(\frac{1}{(1+|Z|^2)^2+t^2}\Bigl)^{\frac{Q-2}{4}}.
$$

We first deal with problem \eqref{P} in the case in which $\Phi$ is
closed to a constant, namely, the perturbation problem
\begin{equation}\label{P1}
-\Delta_{\mathbb{H}^n}u(Z,t)=(1+\va
K(Z,t))u(Z,t)^{\frac{Q+2}{Q-2}},\,\,\,(Z,t)\in \mathbb{H}^n,
\end{equation}
where $\va>0$ is a small parameter, $K(Z,t)=K(|Z|,t)$ is a bounded
cylindrical function on $\mathbb{H}^n$ and satisfies that for some
$\delta>0$
\begin{equation}\label{K1}
\begin{array}{ll}
K(Z,t)=&K(0,\bar{t})+\xi|Z|^{2\gamma}+a|t-\bar{t}|^{\gamma}
+O((|Z|^2,t)-(0,\bar{t})|^{\gamma+\sigma}),\vspace{0.2cm}\\
&(Z,t)\in\{(Z,t)
\mathbb{H}^n\,:\,\rho((Z,t)-(0,\bar{t}))<\delta\,\},
\end{array}
\end{equation}
where $\xi,\,a$, $\gamma$ and $\sigma$ are some constants
depending on $\bar{t}$, $\xi<0,\,a<0$, $\gamma\in (1,\,n)$ and
$\sigma\in (0,\,1)$.

On \eqref{P1}, we have
\begin{theorem}Suppose that $n>1$ and  $K(Z,t)$ satisfies \eqref{K1} in the neighborhood of $(0,\bar{t}^1),\,(0,\bar{t}^2)$, $(\bar{t}^1\neq
\bar{t}^2)$. Then there exists  $\va_0>0$ such that for
 $\va\in(0,\va_0)$, \eqref{P1} has a solution in $S^1_{cy}(\mathbb{H}^n)$ of the form
$$
u_\varepsilon(Z,t)=\sum\limits_{j=1}^2
V_{s^j_\va,\lambda_{\va,j}}(Z,t)+v_\varepsilon(Z,t)
$$
with $\lambda_{\va,j} \to +\infty$, $s^j_\va\to \bar{t}^j$
\,($j=1,2$) and $v_{\va}\in S^1_{cy}(\mathbb{H}^n)$,
$\|v_\va\|_{S_0^1(\mathbb{H}^n)}\to 0$ as $\va\to 0$.
\end{theorem}

We also consider the non-perturbation problem \eqref{P}. In the
following result, we can find some solutions to \eqref{P}
concentrating at exact two points and the distance between these
two points can be very large. Moreover, we construct infinitely
many solutions for \eqref{P} or \eqref{W} under the condition that
$\Phi(\zeta)$ has a sequence of strictly local maximum points
moving to infinity.
\begin{theorem}
Assume that $\Phi(Z,t)=\Phi(|Z|,t)$ is bounded and continuous in
$\mathbb{H}^n\,(n\geq 1)$ and satisfies:

$\Phi(Z,t)$ has a sequence of strictly local maximum point
$(0,\bar{t}^j)\in\mathbb{H}^n$ such that $\rho((0,\bar{t}^j))\to
+\infty$ and in a small neighbourhood of each $(0,\bar{t}^j)$,
there are constants $K_j>0$ and $\gamma_j\in (n,\,n+2)$ such that
\begin{equation}\label{V1}
\Phi(Z,t)=K_j-Q_j((|Z|,t)-(0,\bar{t}^j))+R_j((|Z|,t)-(0,\bar{t}^j)),\,\,\rho((Z,t),(0,\bar{t}^j))<\nu,
\end{equation}
where $K_j$ satisfies $0<c_1\leq K_j\leq c_2<\infty$,
$j=1,\cdot\cdot\cdot,$ $Q_j$ satisfies
$$
a_0(|Z|^4+t^2)^{\frac{\gamma_j}{2}}\leq Q_j(|Z|,t)\leq
a_1(|Z|^4+t^2)^{\frac{\gamma_j}{2}},\,\,\,j=1,\cdot\cdot\cdot,
$$
for some constants $0<a_0\leq a_1<\infty$ independent of $j$, and
$R_j(|Z|,t)$ satisfies
$R_j(|Z|,t)=O((|Z|^4+t^2)^{\frac{\gamma_j+\sigma}{2}})$ for some
$\sigma>0$ independent of $j$.

Then for each small $\mu>0$ and $\bar{t}^{j_1}$, we can find
another strictly local maximum point $\bar{t}^{j_2}$, such that
\eqref{P} has a solution in $S_0^1(\mathbb{H}^n)$ of the form
$$
u=\sum\limits_{l=1}^2
[\Phi(0,\bar{t}^{j_l})]^{-n/2}V_{t^{j_l},\lambda_{j_l}}(Z,t)+v(Z,t),
$$
where
$$
\|v\|_{S_0^1(\mathbb{H}^n)}\leq\mu,\,\,\,|t^{j_l}-\bar{t}^{j_l}|\leq
\mu,\,\,\,|t^{j_1}-t^{j_2}|\geq \frac{1}{\mu},\,\,\,\la_{j_l}\geq
\frac{1}{\mu}.
$$
\end{theorem}

We should point out here that if $\Phi(Z,t)$ is not a constant
identically, our assumption that $\Phi(Z,t)$ has at least at two
points is necessary for the existence of solutions to our
problems. Indeed, if $u$ is a solution to \eqref{P}, then $u$
satisfies the following identity:
$$
\int_{\mathbb{H}^n}\langle (Z,2t),\nabla \Phi(Z,t)\rangle
u^{\frac{2Q}{Q-2}} dZdt=0,
$$
provided the integral is convergent and $K$ is bounded and smooth
(see \cite{Gar}). Hence, $\langle (Z,2t),\nabla \Phi(Z,t)\rangle$
cannot have fixed sign in $\mathbb{H}^n$. As a result, if $\langle
(Z,2t),\nabla \Phi(Z,t)\rangle\geq 0$, then \eqref{K1} or
\eqref{V1} cannot hold in $\mathbb{H}^n$. If $\langle
(Z,2t),\nabla \Phi(Z,t)\rangle\leq 0$, then there are at least two
points such that \eqref{K1} or \eqref{V1} is satisfied.

Solutions obtained in Theorem 1.1 above is two-peaked (that is,
solutions concentrate at exactly two points simultaneously as
$\va\to 0$). However, a direct corollary from our proof of the
theorems is
\begin{corollary}
Under the assumptions of Theorem 1.1 and Theorem 1.2, problem
\eqref{P} does not have single-peaked solution(that is, solutions
concentrate at exactly one point) as $\va\to 0$.
\end{corollary}
In fact, if $u_\va$ concentrates exactly at one point, combining
the fact $\frac{\pa J_\va(\eta,\la,v)}{\pa \la}=0$ and Lemma 4.1
(where the interaction vanishes), we obtain a contradiction
$$
\la_\va^{-\gamma}=o(\la_\va^{-\gamma}).
$$

Our techniques consist in the transformation of the problems first
into a special form of   critical Grushin-type equations and then
into a special form of  Hardy-Sobolev-type equations on the
euclidean space, and the reduction of the problems to a study of a
finite-dimensional functional by a type of Lyapunov-Schmidt
reduction. In fact, we will give some  existence results of
concentration solutions  on  more general Hardy-Sobolev-type
equations. We will see later that it is right the transformation
of the problems into a Hardy-Sobolev-type equations on the
euclidean space that helps us to obtain more precise estimates and
furthermore to obtain more precise solutions.

We summarize the rest of the paper. In Section 2, we transform the
problems into  Hardy-Sobolev-type equations on the euclidean space
and give some more general results on the new equations. In
Section 3 we give some notations and the sketch of the proof of
the main results. The Lyapunov-Schmidt reduction is used to reduce
an infinite system to a finite one. Section 4 is devoted to the
proof of our main results with degree argument and energy analysis
method. For complement, all the basic estimates needed are proved
in Section 5.

\section{En equivalent problem}
In this section, we follow the idea in \cite{Ca} to derive an
equivalent problem of problem \eqref{P} in the cylindrical case.

 Consider the problem \eqref{P}. Recall that the Lie algebra of
the left-invariant vector fields on $\mathbb{H}^n$ is generated by
$$
X_j=\frac{\pa}{\pa x_j}+2y_j\frac{\pa}{\pa
t},\,\,\,\,Y_j=\frac{\pa}{\pa y_j}-2x_j\frac{\pa}{\pa
t},\,\,\,\,j=1,\cdot\cdot\cdot,n.
$$
The sub-elliptic gradient on $\mathbb{H}^n$ is given by
$\nabla_{\mathbb{H}^n}=(X_1,\cdot\cdot\cdot,X_n,Y_1,\cdot\cdot\cdot,Y_n)$
and the Kohn Laplacian on $\mathbb{H}^n$ is the degenerate-elliptic
PDO
$$
\Delta_{\mathbb{H}^n}=\sum\limits_{j=1}^n(X^2_j+Y^2_j).
$$
Then by direct computation one can see that
$$
X_i^2u=\frac{\pa^2 u}{\pa x_i^2}+4y_i\frac{\pa^2 u}{\pa x_i\pa
t}+4y_i^2\frac{\pa^2 u}{\pa t^2},\,\,\,\,\,Y_i^2u=\frac{\pa^2 u}{\pa
y_i^2}-4x_i\frac{\pa^2 u}{\pa y_i\pa t}+4x_i^2\frac{\pa^2 u}{\pa
t^2}.
$$
Hence, if $u(Z,t)=u(|Z|,t)>0$ is cylindrical symmetric (this is
natural in the Heisenberg group $\mathbb{H}^n$), then problem
\eqref{P} becomes
\begin{equation}\label{2.1}
-\Delta_Zu-4|Z|^2u_{tt}=\Phi(|Z|,t)u(|Z|,t)^{\frac{Q+2}{Q-2}},\,\,\,\,u>0,\,\,\,\,(Z,t)\in\R^{2n}\times\R,
\end{equation}
where $\Delta_Z$ is the Eculidean laplacian in $\R^{2n}$.

Equation (2.1) is a special form of the following problem related to
the Grushin operator
\begin{equation}\label{G}
\mathbb{G}u\triangleq-\Delta_yu-4|y|^2u_{z}=\Phi(y,z)u(y,z)^{\frac{Q+2}{Q-2}},\,\,\,\,u>0,\,\,\,\,\,(y,z)\in\R^{m_1}\times\R^{m_2},
\end{equation}
where $Q=m_1+2m_2$ is the ``appropriate" dimension and
$\frac{Q+2}{Q-2}$ is the corresponding critical exponent.

If $\Phi=\Phi(|y|,z)$ and $u=\psi(|y|,z)$ satisfy problem
\eqref{G}, then
\begin{equation}\label{2.3}
-\psi_{rr}(r,z)-\frac{m_1-1}{r}\psi_r(r,z)-4r^2\Delta_z\psi(r,z)=\Phi(|y|,z)\psi(r,z)^{\frac{Q+2}{Q-2}},
\end{equation}
where $r=|y|$.

Define
$$
v(r,z)=\psi(\sqrt{r},z).
$$
Then
$$
\psi_r(\sqrt{r},z)=2\sqrt{r}v_r(r,z),\,\,\,\psi_{rr}(\sqrt{r},z)=4rv_{rr}(r,t)+2v_r(r,z).
$$
Hence $v$ satisfies
\begin{equation}\label{2.4}
-v_{rr}(r,z)-\frac{m_1}{2r}v_r(r,z)-\Delta_z\psi(r,z)=\frac{\Phi(\sqrt{r},z)}{4r}v(r,z)^{\frac{Q+2}{Q-2}},
\end{equation}
that is, $v=v(|y|,z)$ solves the following Hardy-Sobolev-type
problem
\begin{equation}\label{2.5}
-\Delta
u(y,z)=\phi(y,z)\frac{u^{\frac{k+h}{k+h-2}}}{|y|},\,\,\,\,(y,z)\in\R^k\times\R^h,
\end{equation}
where $k=\frac{m_1+2}{2}$, $h=m_2$ and
$\phi(y,z)=\phi(|y|,z)=\frac{\Phi(\sqrt{r},z)}{4}$.

As a result, we can summarize the above facts to conclude that
\begin{proposition}
Let $m_1$ be even and $\Phi(y,z)=\Phi(|y|,z)$, then
$u(y,z)=u(|y|,z)$ solves problem \eqref{G} if and only if
$v(y,z)=u(\sqrt{|y|},z)$ solves problem \eqref{2.5} with
$k=\frac{m_1+2}{2}$, $h=m_2$ and
$\phi(y,z)=\phi(|y|,z)=\frac{\Phi(\sqrt{r},z)}{4}$. In particular,
$u(\zeta)=u(|Z|,t)$ solves problem \eqref{P} if and only if
$v(|Z|,t)=u(\sqrt{|Z|},t)$ solves problem \eqref{2.5} with $k=n+1$,
$h=1$. Moreover, there exists $c_{n}>0$ such that
\begin{equation}\label{E}
\int_{\mathbb{H}^n}|\nabla_{\mathbb{H}^n}u|^2dZdt=
c_n\int_{\R^k\times\R}|\nabla v|^2dydt.
\end{equation}
\end{proposition}
\begin{proof}
We only prove \eqref{E}. This can be done by the following
calculation:
\begin{eqnarray*}
\int_{\mathbb{H}^n}|\nabla_{\mathbb{H}^n}u|^2dZdt&=&\int_{\mathbb{H}^n}\sum\limits_{i=1}^n
(|X_i(u)|^2+|Y_i(u)|^2)dZdt\\
&=&\int_{\mathbb{H}^n}\sum\limits_{i=1}^n \Bigl(\bigl|\frac{\pa
u}{\pa x_i}\bigl|^2+\bigl|\frac{\pa u}{\pa
y_i}\bigl|^2+4(x_i^2+y_i^2)\bigl|\frac{\pa u}{\pa
t}\bigl|^2\Bigl)dZdt\\
&=& \om_{2n}\int_{\R^+\times\R} \Bigl(\bigl|\frac{\pa u}{\pa
r}\bigl|^2+4r^2\bigl|\frac{\pa u}{\pa t}\bigl|^2\Bigl)r^{2n-1}drdt\\
&=&2\om_{2n}\int_{\R^+\times\R} \Bigl(\bigl|\frac{\pa v}{\pa
r}\bigl|^2+\bigl|\frac{\pa v}{\pa t}\bigl|^2\Bigl)r^{k-1}drdt\\
&=&\frac{2\om_{2n}}{\om_{k}}\int_{\R^k\times\R}|\nabla v|^2dydt,
\end{eqnarray*}
where (and in the sequel) $\om_N$ is the measure of the $N-1$
dimensional sphere $S^{N-1}$.
\end{proof}

In the sequel, we will consider a more general problem, that is
\begin{equation}\label{F}
-\Delta u(y,z)=\phi(y,z)\frac{u^{\frac{N}{N-2}}}{|y|},
\,\,\,\,u>0,\,\,\,(y,z)\in\R^k\times\R^h=\R^N,\,\,(k\geq 2,\,h\geq
1).
\end{equation}
 Consider the limiting problem
\begin{equation}\label{L}
-\Delta u=\frac{u^{\frac{N}{N-2}}}{|y|},\,\,\,u>0,\,\,\,
x\triangleq(y,z)\in\,\R^N, u\in D^{1,2}(\R^N),
\end{equation}
where
$$
\D=\{u\in L^{2(N-1)/(N-2)}(|y|,\R^N):\,\,|\nabla u|\in L^2(\R^N)\}
$$
and $\D$ endows the norm $\|u\|\triangleq(\int_{\N}|\nabla
u|^2dx)^{1/2}$, which is induced by the inner produce $\langle
u,v\rangle=\int_{\N} \nabla u \nabla v dx$. It is known from
\cite{Ca} that for $\zeta \in \R^h$, $\la
>0$, functions
$$
U_{\zeta,
\la}(x)=\frac{[(N-2)(k-1)]^{\frac{N-2}{2}}\la^{\frac{N-2}{2}}}
{\Bigl((1+\la |y|)^2+\la^2|z-\zeta|^2\Bigl)^{\frac{N-2}{2}}}
$$
solve \eqref{L}.

\begin{corollary}
$U_{0,1}=2^{-n}V_{0,1}$ and $c_0=(2n)^n$, where $c_0$ and
$V_{0,1}$ are  defined by \eqref{M}.
\end{corollary}
\begin{proof}
By \eqref{L1}, \eqref{E} and \eqref{L}, we can deduce from
Proposition 2.1 that $U_{0,1}=2^{-n}V_{0,1}$. Moreover,  by direct
calculation, we see $c_0=(2n)^n$.
\end{proof}

We first consider the case in which $\phi(y,z)$ is a perturbation
of a constant, that is $\phi(y,z)=1+\va  K(y,z)$. Suppose that for
some $\delta>0$
\begin{equation}\label{K}
\begin{array}{ll}
K(x)=&K(0,\bar{\eta})+\sum\limits_{i=1}^{k}\xi_i|y|^\gamma+\sum\limits_{i=1}^{h}a_i|z_i-\bar{\eta}_i|^\gamma
+O(|x-(0,\bar{\eta}|^{\gamma+\sigma}),\vspace{0.2cm}\\
&\,\,\,x\in\{x\in \R^N(N>3)\,:\,|x-(0,\bar{\eta})|<\delta\,\},
\end{array}
\end{equation}
where $\xi_i,\,a_j$, $\gamma$ and $\sigma$ are some constants
depending on $\bar{\eta}$, $\xi_i,\,a_j\neq 0$ for
$i=1,\cdot\cdot\cdot,k,\,\,j=1,\cdot\cdot\cdot,h$, $\gamma\in
(1,\,N-2)$ and $\sigma\in (0,\,1)$.
 Set
 $\xi=(\xi_1,\cdot\cdot\cdot,\xi_k),\,\,\mathbf{a}=(a_1,\cdot\cdot\cdot,a_h)$.
 Define
$$
g(\pi_1,\pi_2,\gamma,\xi,\mathbf{a})=\frac{\pi_1}{k}\sum_{j=1}^{k}\xi_j
+\frac{\pi_2}{h}\sum_{j=1}^{h}a_j,
$$
where
\begin{eqnarray*}
\pi_1=\int_{\N}
\frac{|y|^{\gamma}(1-|y|^2-|z|^2)dx}{|y|[(1+|y|)^2+|z|^2]^{N}},\,\,\,\,\,\pi_2=\int_{\N}
\frac{|z|^{\gamma}(1-|y|^2-|z|^2)dx}{|y|[(1+|y|)^2+|z|^2]^{N}}.
\end{eqnarray*}
We remark that by Lemma 5.9,  $\pi_1<0,\,\pi_2<0$.

Suppose that
\begin{equation}\label{g}
g(\pi_1,\pi_2,\gamma,\xi,\mathbf{a})>0.
\end{equation}

Define
\[
\begin{array}{ll}
\Lambda:=\Bigl\{(0,\bar{\eta})\in\R^{N}&:\,\,D_{x}K(x)\bigl|_{x=(0,\bar{\eta})}=0,\,\,\,K(x)
\,\,\hbox{satisfies}\,\,\,\eqref{K}\,\,\,\hbox{and}\,\,\,\,\eqref{g}\Bigl\}.
\end{array}
\]

 The following result is corresponding to Theorem 1.1.
\begin{theorem}Suppose that $K(y,z)$ is bounded and continuous in $\R^N\,(N>3)$, $\phi(y,z)=1+\va K(y,z)$, $\La$ contains at least two points.
Then for each $\bar{\eta}^1,\,\bar{\eta}^2\in\La$,
 $\bar{\eta}^1 \neq \bar{\eta}^2$, there exists  $\va_0>0$ such that for
 $\va\in(0,\va_0)$, \eqref{F} has a solution of the form
$$
u_\varepsilon=\sum\limits_{j=1}^2
U_{\eta^j_\va,\lambda_{\va,j}}+v_\varepsilon
$$
with $\lambda_{\va,j} \to +\infty$, $\eta^j_\va\to \bar{\eta}^j$
\,($j=1,2$) and $\|v_\va\|\to 0$  as $\va\to 0$.
\end{theorem}

We also construct some solutions to \eqref{F} which concentrate
exactly at two  points between which the distance can be very
large. This result is a counterpart of Theorem 1.3.

\begin{theorem}
Assume that $\phi$ is bounded and continuous in $\N$ and
satisfies:

$\phi(y,z)$ has a sequence of strictly local maximum point
$(0,\bar{\eta}^j)\in\R^N\,(N\geq 3)$ such that $|\bar{\eta}^j|\to
+\infty$ and in a small neighbourhood of each $\bar{\eta}^j$,
there are constants $K_j>0$ and $\gamma_j\in (N-2,\,N)$ such that
\begin{equation}\label{V}
\phi(x)=K_j-Q_j(x-(0,\bar{\eta}^j))+R_j(x-(0,\bar{\eta}^j)),\,\,x=(y,z)\in
B_\nu(0,\bar{\eta}^j),
\end{equation}
where $K_j$ satisfies $0<c_1\leq K_j\leq c_2<\infty$,
$j=1,\cdot\cdot\cdot,$ $Q_j$ satisfies
$$
a_0|x|^{\gamma_j}\leq Q_j(x)\leq
a_1|x|^{\gamma_j},\,\,\,j=1,\cdot\cdot\cdot,
$$
for some constants $0<a_0\leq a_1<\infty$ independent of $j$, and
$R_j(x)$ satisfies $R_j(x)=O(|x|^{\gamma_j+\sigma})$ for some
$\sigma>0$ independent of $j$.

Then for each small $\mu>0$ and $\bar{\eta}^{j_1}$, we can find
another strictly local maximum point $\bar{\eta}^{j_2}$, such that
\eqref{F} has a solution of the form
$$
u=\sum\limits_{l=1}^2 K_{j_l}^{(2-N)/2}U_{\eta^{j_l},\la_{j_l}}+v,
$$
where
$$
\|v\|\leq\mu,\,\,\,|\eta^{j_l}-\bar{\eta}^{j_l}|\leq
\mu,\,\,\,|\eta^{j_1}-\eta^{j_2}|\geq
\frac{1}{\mu},\,\,\,\la_{j_l}\geq \frac{1}{\mu}.
$$
\end{theorem}

\section{Notations and preliminary results}
 The functional corresponding to \eqref{F}  can be
defined as
$$
I(u)=\frac{1}{2}\int_{\N}|\nabla
u|^2dx-\frac{N-2}{2(N-1)}\int_{\N}\phi(y,z)\frac{|u|^{\s}}{|y|}dx,\,\,\,u\in
\D.
$$

In what follows, we mainly concentrate on the case
$\phi(y,z)=1+\va K(y,z)$. Since the case for non-perturbation in
Theorem 2.4 is similar, we will give a sketch to the proof of
Theorem 2.4 in Section 4.

 We will restrict our arguments to the existence of that
particular solution of \eqref{F} that concentrates, as $\va\to 0$,
at $\bar{\eta}^1,\bar{\eta}^2$, that is a solution of the form
$$
u_\varepsilon=\sum\limits_{j=1}^2
U_{\eta^j_\va,\lambda_{\va,j}}+v_\varepsilon
$$
with $\lambda_{\va,j} \to +\infty$, $\eta^j_\va\to \bar{\eta}^j$
\,($j=1,2$) and $\|v_\va\|\to 0$  as $\va\to 0$.

For $\eta=(\eta^1,\eta^2)\in \R^{h}\times\R^{h}$,
$\la=(\la_1,\la_2)\in \R_+\times\R_+$,  denote
\begin{equation}\label{1.7}
\begin{array}{ll}
  E_{\eta, \la}=\bigg\{v \in \D :  \bigg\langle \ds\frac{\pa \Uj}{\pa
  \la_j},v \bigg\rangle
  = \bigg\langle\ds\frac{\pa \Uj}{\pa
  \eta^j_i},v\bigg\rangle=0,\\
  \hspace{4cm}\,\,\,for\,\,\, j=1,2,\,\,\,i=1,\cdots,h
  \bigg\}.
\end{array}
\end{equation}

For each $(0,\bar{\eta}^1),\,(0,\bar{\eta}^2)\in\La$,
$\bar{\eta}^1\neq\bar{\eta}^2$, $\mu>0$, set
\begin{eqnarray*}
 D_\mu&=&\bigg\{(\eta,\la):\,\eta=(\eta^1,\eta^2)\in\overline{B_\mu(\bar{\eta}^1)}\times\overline{B_\mu(\bar{\eta}^2)}
 \subset \R^h\times\R^h,\,\,\\
 &&\hspace{3cm}\la=(\la_1,\la_2)\in\Bigl(\frac{1}{\mu},+\infty\Bigl)\times\Bigl(\frac{1}{\mu},+\infty\Bigl)
  \bigg\},\\
  M_\mu&=&\bigg\{(\eta,\la,v): (\eta,\la) \in D_\mu,\,\, v \in
  E_{\eta,\la},\,\,\|v\|< \mu \bigg\}.
\end{eqnarray*}

Let
\begin{equation}
  J(\eta,\la,v)=I(\sum\limits_{j=1}^2 \Uj+v).
\end{equation}
Now similar  to \cite{Ba,Rey1}, we have the following lemma.
\begin{lemma}
For $\mu>0$ small, $ u=\sum\limits_{j=1}^2\Uj+v $ is a positive
critical point of $I(u)$ in $\D$ if and only if $(\eta,\la,v)$ is
a critical point of $J(\eta,\la,v)$ in $M_\mu$.
\end{lemma}

 On the other hand, it
follows from Lagrange multiplier theorem that $(\eta,\la,v) \in
M_\mu$ is a critical point of $J(\eta,\la,v)$ in the manifold
$M_\mu$ if and only if there are numbers $B_j \in \R$,  $C_{ji}
\in \R $ for $i=1,\cdots,h, j=1,2$ such that
\begin{equation}\label{2.11}
   \frac{\pa J(\eta,\la,v)}{\pa v}=\sum\limits_{j=1}^2B_j \frac{\pa \Uj}{\pa
  \la_j}+\sum_{j=1}^2\sum\limits_{i=1}^{h} C_{ji}\frac{\pa \Uj}
  {\pa \eta_i^j},
\end{equation}
\begin{equation}\label{1.11}
  \frac{\pa J (\eta,\la,v)}{\pa \la_j}
  =B_j\bigg\langle\frac{\pa^2 \Uj}{\pa
  \la_j^2},v \bigg\rangle+\sum_{l=1}^{h} C_{jl} \bigg\langle
  \frac{\pa^2 \Uj}{\pa \la_j \pa \eta^j_l},v
  \bigg\rangle,j=1,2,
\end{equation}
\begin{equation}\label{2.10}
\begin{array}{ll}
 \ds\frac{\pa J(\eta,\la,v)}{\pa \eta_i^j}=B_j \bigg\langle
 \ds\frac{\pa^2 \Uj}{\pa
  \la_j \pa \eta_i^j},v \bigg\rangle+& \sum\limits_{l=1}^{h} C_{jl} \bigg\langle
  \ds\frac{\pa^2 \Uj}{\pa \eta_l^j \pa \eta_i^j},v
  \bigg\rangle,\,\,j=1,2,\,i=1,\cdot\cdot\cdot,h.
\end{array}
\end{equation}

In order to verify Theorem 1.1, following the ideas of
\cite{Rey1}, we show first that for  $(\eta,\la)\in D_\mu$ given,
there exists $v\in E_{\eta, \la}$ and scalars
$B_j,C_{ji},i=1,\cdots, h, j=1,2$
 such that \eqref{2.11} is satisfied. Then as in \cite{CNY3}, we employ a degree argument to find
 suitable $(\eta,\la)\in D_\mu$ such that \eqref{1.11}, \eqref{2.10} are
satisfied.

Throughout this paper we will let
$\va_{ij}=(\la_i\la_j|\eta^i-\eta^j|^2)^{(2-N)/2}$ for $i\neq j$
and $C_{N,k}=[(N-2)(k-1)]^{N-1}$.

\begin{proposition}\label{3.2}For $\bar{\eta}^1,\,\bar{\eta}^2\in\La$ and $(\eta,\la)\in
D_\mu$, there exist $\va_0>0$, $\mu_0>0$ and a $C^1-$map which, to
any $(\eta,\la)\in D_\mu$,  $\va \in (0,\va_0]$, $\mu \in
(0,\mu_0]$, associates $v_\va(\eta,\la):D_\mu\to E_{\eta,\la}$
such that $v_\va(\eta,\la)$ satisfies \eqref{2.11}  for some
$B_j,C_{ji} (i=1,\cdot\cdot\cdot,h, j=1,2)$. Furthermore,
$v_\va(\eta,\la)$ satisfies the following estimate as $\va\to 0$
$$
\|v_\va(\eta,\la)\|=O\Bigl(\va\sum\limits_{j=1}^2\Bigl(|\eta^j-\bar{\eta}^j|^{\gamma_j}
+\frac{1}{\la_j^{\gamma_j}}\Bigl)\Bigl)+O(\va_{12}^{\frac{1}{2}
+\tau}),
$$
where $\tau>0$ is some constant.
\end{proposition}
\begin{proof}
We expand $J(\eta,\la,v)$ in the neighborhood $v=0$. For $v\in
E_{\eta,\la}$ we obtain
\begin{equation}\label{2.1}
  J(\eta,\la,v)=J(\eta,\la,0)+ \langle f_\va,v\rangle
  +\frac{1}{2}\langle Q_\va v,v\rangle+R_\va(v),
\end{equation}
where $f_\va \in E_{\eta,\la}$ is the linear form over
$E_{\eta,\la}$ given by
\begin{equation}
\begin{array}{ll}
 \langle f_\va,v\rangle =&\langle \sum\limits_{j=1}^{2}\Uj,v\rangle-\ds\int_{\N}\frac{1+\va
 K(x)}{|y|}
 \bigg(\sum\limits_{j=1}^2\Uj\bigg)
 ^{\frac{N}{N-2}}vdx,
\end{array}
\end{equation}
$\langle Q_\va v,v\rangle$ is the quadratic form on $E_{\eta,\la}$
given by
\begin{equation}\label{2.3}
\begin{array}{ll}
\langle Q_\va
v,v\rangle=&\|v\|^2-\ds\frac{N}{N-2}\ds\int_{\N}\frac{1+\va
K(x)}{|y|}
 \bigg(\sum\limits_{j=1}^2\Uj\bigg)
 ^{\frac{2}{N-2}}v^2dx,
\end{array}
\end{equation}
and $R_\va(v)$ is the  higher order term satisfying
$$
D^{(i)}R_\va(v)=O(\|v\|^{2+\theta-i}),\,\,i=1,2,
$$
where $\theta>0$ is some constant.

From Proposition 3.3, $Q_\va$ is invertible and $\|Q_\va^{-1}\|\leq
C$ for some $C>0$ independent of $\eta,\la$ and $\va$. Now following
the arguments in \cite{CNY3,Rey1} we have
$$
\frac{\pa J(\eta,\la,v)}{\pa v}\bigg|_{ E_{\eta,\la}}= f_\va+Q_\va
v+DR_\va(v).
$$
There exists an equivalence between the existence of $v$ such that
\eqref{2.11} holds for $(\eta,\la,v)$ and
\begin{equation}\label{2.4}
  f_\va+Q_\va v+DR_\va(v)=0.
\end{equation}
We are now in the position to use the argument  in \cite {Rey1} to
establish the existence of $v(\eta,\la)$ such that (3.3) is
satisfied for some  numbers $B_j,C_{ji} (i=1,\cdot\cdot\cdot,h,
j=1,2)$. Moreover, there exists a constant $C>0$ such that
\begin{equation}
  \|v\|\leq C\|f_\va\|.
\end{equation}

Now we estimate $\|f_\va\|$. Note that
$$
\begin{array}{ll}
\langle
f_\va,v\rangle=&\ds\int_{\N}\frac{1}{|y|}\Bigl(\sum\limits_{j=1}^2\Uj^{\frac{N}{N-2}}-(\sum\limits_{j=1}^2\Uj)^{\frac{N}{N-2}}
\Bigl)vdx\\
&-\va\ds\int_{\N}\frac{K(x)}{|y|}\Bigl(\sum\limits_{j=1}^2\Uj\Bigl)^{\frac{N}{N-2}}vdx.
\end{array}
$$
On the other hand, by H\"{o}lder inequality and Lemma 5.1 in the
appendix,
\begin{eqnarray*}
&&\Bigl|\int_{\N}\frac{1}{|y|}\Bigl(\sum\limits_{j=1}^2\Uj^{\frac{N}{N-2}}-(\sum\limits_{j=1}^2\Uj)^{\frac{N}{N-2}}
\Bigl)vdx\Bigl|\\
&=&\left\{\begin{array}{ll} O\Bigl(\sum\limits_{i\neq
j}\ds\int_{\N}\frac{1}{|y|}\Ui^{\frac{N}{2(N-2)}}\Uj^{\frac{N}{2(N-2)}}|v|dx\Bigl)\,\,\,\,(2<\s<3)\\
O\Bigl(\sum\limits_{i\neq
j}\ds\int_{\N}\frac{1}{|y|}\Ui^{\frac{2}{N-2}}\Uj|v|dx\Bigl)\,\,\,\,(\s\geq
3)
\end{array}
\right.\\
&=& \left\{\begin{array}{ll} O\Bigl(\sum\limits_{i\neq
j}\Bigl(\ds\int_{\N}\frac{1}{|y|}\Ui^{\frac{N-1}{N-2}}\Uj^\frac{N-1}{N-2}dx\Bigl)^{\frac{N}{2(N-1)}}\Bigl)\|v\|\,\,\,\,(2<\s<3)\\
O\Bigl(\sum\limits_{i\neq
j}\Bigl(\ds\int_{\N}\frac{1}{|y|}\Ui^{\frac{4(N-1)}{N(N-2)}}\Uj^{\frac{2(N-1)}{N}}dx\Bigl)
^{\frac{N}{2(N-1)}}\Bigl)\|v\| \,\,\,\,(\s\geq 3)
\end{array}
\right.\\
 &=&O(\va_{12}^{\frac{1}{2} +\tau})\|v\|.
\end{eqnarray*}
Similarly,
\begin{eqnarray*}
&&\Bigl|\ds\int_{\N}\frac{K(x)}{|y|}\Bigl(
\sum\limits_{j=1}^2\Uj\Bigl)^{\frac{N}{N-2}}vdx\Bigl|\\
&=&\Bigl|\sum\limits_{j=1}^2\ds\int_{\N}\frac{K(x)}{|y|}\Uj^{\frac{N}{N-2}}vdx\Bigl|+O(\va_{12}^{\frac{1}{2}
+\tau})\|v\|\\
&=&O\Bigl(\sum\limits_{j=1}^2\Bigl(
|\eta^j-\bar{\eta}^j|^{\gamma_j}+\frac{1}{\la_j^{\gamma_j}}\Bigl)\Bigl)\|v\|+O(\va_{12}^{\frac{1}{2}
+\tau})\|v\|.
\end{eqnarray*}
As a result, combining the above three equations, we complete the
proof.
\end{proof}

\begin{proposition}\label{2.1}Let $(\eta,\la) \in D_\mu$. Then for $\mu>0$ ,
$\va>0$ sufficiently small, there exists a $\rho>0$ such that
$$
\|Q_\va \om\|\geq\rho\|\om\|,\,\,\,\,\forall\,\,\,\om\in
E_{\eta,\la}.
$$
\end{proposition}
\begin{proof}
By the boundedness of $K(x)$, it suffices to prove  the
proposition for the case $\va=0$. The main idea of the proof is
similar to that of Lemma 2.3 in \cite{Li}.

We argue by contradiction. Suppose that there are $\mu_n\to 0$,
$(\eta^n,\la_n)\in D_{\mu_n}$ and $\om_n\in E_{\eta^n,\la_n}$ such
that
\begin{equation}\label{h}
\|Q_0\om_n\|=o(1)\|\om_n\|,
\end{equation}
where $o(1)\to 0$ as $n\to\infty$. In \eqref{h}, we may assume
$\|\om_n\|=1$.

For $j=1,2$, let
$\tilde{\om}_{j,n}(x)=\la^{(2-N)/2}_{j,n}\om_n(\la_{j,n}^{-1}x+(0,\eta^{j,n}))$.
Then $\tilde{\om}_{j,n}(x)$ is bounded in $\D$. Hence we may
assume that there is $\om_j\in\D$ such that as $n\to\infty$,
\begin{equation}\label{i}
\tilde{\om}_{j,n}\rightharpoonup
\om_j,\,\,\,\,\,\hbox{weakly}\,\,\,\,\hbox{in}\,\,\,\,\D.
\end{equation}

Now we verify that $\om_j=0$.

Define
\begin{eqnarray*}
\tilde{U}_{j,n}&=&\la^{(2-N)/2}_{j,n}U_{\eta^{j,n},\la_{j,n}}(\la_{j,n}^{-1}x+(0,\eta^{j,n}))\\
W_{j,n,i}&=&\la^{-N/2}_{j,n}\frac{\pa
U_{\eta^{j,n},\la_{j,n}}(P)}{\pa\eta^{j,n}_i}\bigl|_{P=\la_{j,n}^{-1}x+(0,\eta^{j,n})},\,\,\,i=1,\cdot\cdot\cdot,h,\\
W_{j,n}&=&\la^{(4-N)/2}_{j,n}\frac{\pa
U_{\eta^{j,n},\la_{j,n}}(P)}{\pa\la^{j,n}}\bigl|_{P=\la_{j,n}^{-1}x+(0,\eta^{j,n})}.
\end{eqnarray*}
$\om_n\in E_{\eta^n,\la_n}$ implies that
$$
\tilde{\om}_{j,n}\in \tilde{E}_{n}\triangleq
\Bigl\{\varphi\in\D:\,\,\,\int_{\N}\nabla W_{j,n}\nabla\varphi
dx=\int_{\N}\nabla W_{j,n,i}\nabla\varphi dx=0\Bigl\},
$$
$j=1,2,\,\,i=1,\cdot\cdot\cdot,h$, and
\begin{equation}\label{i}
\begin{array}{ll}
&\ds\int_{\N} \nabla \tilde{\om}_{j,n} \nabla\varphi
dx-\ds\frac{N}{N-2}\ds\int_{\N}\frac{1}{|y|}
  \tilde{U}_{j,n}
  ^{\frac{2}{N-2}}\tilde{\om}_{j,n} \varphi
  dx\\
=&o(1)\|\varphi\|,\,\,\forall\,\,\, \varphi\in \tilde{E}_{n}.
\end{array}
\end{equation}

We claim that $\om_j$ solves
\begin{equation}\label{j}
-\Delta
\om_j-\frac{N}{N-2}\frac{U_{0,1}^{\frac{2}{N-2}}}{|y|}\om_j=0.
\end{equation}
Indeed, for any $\varphi\in C_0^\infty(\N)$, we can find $c_{j,n}$
and $c_{j,n,i}$ such that
$$
\varphi_n=\varphi-\sum\limits_{j=1}^2\sum\limits_{i=1}^h
c_{j,n,i}W_{j,n,i}-\sum\limits_{j=1}^2c_{j,n}W_{j,n}\in\tilde{E}_n.
$$
Since $\varphi$ has compact support and the support of $W_{l,n,i}$
and $W_{l,n}$ moves to infinity as $n\to\infty$ for $l\neq j$, it
is easy to check that $c_{l,n,i}\to 0$, $c_{l,n}\to 0$ as $n\to
\infty$ for any $l\neq j$. Moreover, $c_{j,n,i}$ and $c_{j,n}$ are
bounded.

Inserting $\varphi_n$ into \eqref{i} and letting $n\to\infty$, we
see
\begin{equation}\label{k}
\begin{array}{ll}
&\ds\int_{\N}\nabla\om_j \nabla \varphi
dx-\frac{N}{N-2}\int_{\N}\frac{U_{0,1}^{\frac{2}{N-2}}}{|y|}\om_j\varphi
dx\\
&-c\ds\int_{\N}\Bigl[ \nabla\om_j \nabla \Bigl(\frac{\pa
U_{0,1}}{\pa\la}\Bigl|_{\la=1}\Bigl)
-\frac{N}{N-2}\frac{U_{0,1}^{\frac{2}{N-2}}}{|y|}\om_j\Bigl(\frac{\pa
U_{0,1}}{\pa\la}\Bigl|_{\la=1}\Bigl)\Bigl] dx\\
&-\sum\limits_{i=1}^2c_i \Bigl[\ds\int_{\N} \nabla\om_j \nabla
\Bigl(\frac{\pa U_{0,1}}{\pa \eta_i}\Bigl|_{\eta=0}\Bigl)
dx-\frac{N}{N-2}\frac{U_{0,1}^{\frac{2}{N-2}}}{|y|}\om_j\Bigl(\frac{\pa
U_{0,1}}{\pa\eta_i}\Bigl|_{\eta=0}\Bigl)\Bigl] dx=0,
\end{array}
\end{equation}
where $c=\lim_{n\to\infty}c_{j,n}$ and
$c_i=\lim_{n\to\infty}c_{j,n,i}$. On the other hand,
\begin{eqnarray*}
&&\int_{\N}\Bigl[ \nabla\om_j \nabla \Bigl(\frac{\pa
U_{0,1}}{\pa\la}\Bigl|_{\la=1}\Bigl)
-\frac{N}{N-2}\frac{U_{0,1}^{\frac{2}{N-2}}}{|y|}\om_j\Bigl(\frac{\pa
U_{0,1}}{\pa\la}\Bigl|_{\la=1}\Bigl)\Bigl] dx=0\\
&&\int_{\N} \nabla\om_j \nabla \Bigl(\frac{\pa U_{0,1}}{\pa
\eta_i}\Bigl|_{\eta=0}\Bigl)
dx-\frac{N}{N-2}\frac{U_{0,1}^{\frac{2}{N-2}}}{|y|}\om_j\Bigl(\frac{\pa
U_{0,1}}{\pa\eta_i}\Bigl|_{\eta=0}\Bigl)\Bigl] dx=0.
\end{eqnarray*}
Therefore, we obtain
$$
\int_{\N}\nabla\om_j \nabla \varphi
dx-\frac{N}{N-2}\int_{\N}\frac{U_{0,1}^{\frac{2}{N-2}}}{|y|}\om_j\varphi
dx=0,\,\,\,\,\forall \,\,\varphi\in C^\infty_0(\N),
$$
which implies \eqref{j}.

Proceeding as done to prove  \eqref{k}, we see from
$\tilde{\om}_{j,n}\in\tilde{E}_n$ that
$$
\Bigl\langle\om_j,\frac{\pa
U_{0,1}}{\pa\la}\Bigl|_{\la=1}\Bigl\rangle=\Bigl\langle\om_j,\frac{\pa
U_{0,1}}{\pa\eta_i}\Bigl|_{\eta=0}\Bigl\rangle=0,\,\,\,i=1,\cdot\cdot\cdot,h.
$$
But, it is verified in \cite{Ca} that $U_{0,1}$ is non-degenerate.
As a result, we can conclude that $\om_j=0$.

Now since $\om_j=0$, we see that
\begin{eqnarray*}
\int_{\N}\frac{1}{|y|}
  \bigg(\sum\limits_{j=1}^2 U_{\eta^{j,n},\la_{j,n}}
  \bigg)^{\frac{2}{N-2}}\om_n^2dx
&\leq& \sum\limits_{j=1}^2\int_{\N\setminus
B_{R/\la_{j,n}}(0,\eta^{j,n})}\frac{1}{|y|}
  U_{\eta^{j,n},\la_{j,n}}^{\frac{2}{N-2}}\om_n^2dx+o(1)\\
  &=&o_R(1)+o(1),
\end{eqnarray*}
where $o_R(1)\to 0$ as $R\to+\infty$. Hence \eqref{h} implies that
$$
\int_{\N}|\nabla \om_n|^2dx\to 0,\,\,\,(n\to\infty).
$$
This is a contradiction to $\|\om_n\|=1$.
\end{proof}

\section{Proof of the main results}
To solve system \eqref{1.11} and \eqref{2.10}, we need to estimate
each term in them.
\begin{lemma}\label{3.1}  Let $(\eta,\la) \in D_\mu$,
$v(\eta,\la)$ be obtained in Proposition 3.2. For $\mu
>0$ and $\va >0$ small enough, we have for $i=1,2,$
\begin{eqnarray*}
 \frac{\pa J(\eta,\la,v)}{\pa \la_i}
 &=&-C_{N,k}b_1\frac{\va_{12}}{\la_i}-\frac{(N-2)C_{N,k}}{2\la_i^{\gamma_i+1}}
\Bigl[\frac{b^i_2}{k}\sum_{j=1}^{k}\xi^i_j
+\frac{b^i_3}{h}\sum_{j=1}^{h}a^i_j\Bigl]\va+O\Bigl(\frac{\va\va_{12}}{\la_i}\Bigl)\\
&&+O\Bigl(\frac{\va^{1+\tau}_{12}}{\la_i}\Bigl)
+O\Bigl(\frac{\va}{\la_i}\sum\limits_{j=1}^2
\Bigl(\frac{1}{\la_j^{\gamma_j+\sigma}}
+|\eta^j-\bar{\eta}^j|^{\gamma_j+\sigma}\Bigl)\Bigl)+O\Bigl(\frac{\va}{\la_i^{\gamma_i}}|\eta_i-\bar{\eta}^i|\Bigl),
\end{eqnarray*}
where, $b_1<0,\,b_2^i<0,\,b_3^i<0$ are defined in Lemmas 5.2 and
5.3 in the Appendix.
\end{lemma}
\begin{proof} By direct computation and Lemmas 5.2-5.4, we have
 \begin{eqnarray*}
 &&\frac{\pa J(\eta,\la,v)}{\pa \la_i}\\
 &=&\Bigl\langle\sum\limits_{j=1}^2\Uj+v,\frac{\pa\Ui}{\pa\la_i}\Bigl\rangle\\
 &&-\int_{\N}\frac{1+\va K(x)}{|y|}\Bigl|\sum\limits_{j=1}^2\Uj+v\Bigl|^{\frac{2}{N-2}}
 \Bigl(\sum\limits_{j=1}^2\Uj+v\Bigl)
 \frac{\pa\Ui}{\pa\la_i}dx\\
 &=&\sum\limits_{j=1}^2\Bigl\langle\Uj,\frac{\pa\Ui}{\pa\la_i}\Bigl\rangle
-\int_{\N}\frac{1+\va
h(x)}{|y|}\Bigl(\sum\limits_{j=1}^2\Uj\Bigl)^{\frac{N}{N-2}}\frac{\pa\Ui}{\pa\la_i}dx\\
&&-\frac{N}{N-2}\int_{\N}\frac{1+\va
h(x}{|y|}\Bigl(\sum\limits_{j=1}^2\Uj\Bigl)^{\frac{2}{N-2}}\frac{\pa\Ui}{\pa\la_i}vdx
+O\Bigl(\frac{1}{\la_i}\Bigl)\|v\|^2\\
&=&\int_{\N}\frac{1}{|y|}\Bigl(\sum\limits_{j=1}^2\Uj^{\frac{N}{N-2}}
-\Bigl(\sum\limits_{j=1}^2\Uj\Bigl)^{\frac{N}{N-2}}\Bigl)\frac{\pa\Ui}{\pa\la_i}dx\\
&&-\va\int_{\N}
\frac{K(x)}{|y|}\Bigl(\sum\limits_{j=1}^2\Uj\Bigl)^{\frac{2}{N-2}}\frac{\pa\Ui}{\pa\la_i}dx\\
&&-\frac{N}{N-2}\int_{\N}\frac{1+\va
K(x)}{|y|}\Bigl(\sum\limits_{j=1}^2\Uj\Bigl)^{\frac{2}{N-2}}\frac{\pa\Ui}{\pa\la_i}vdx
+O\Bigl(\frac{1}{\la_i}\Bigl)\|v\|^2\\
&=&-\frac{N}{N-2}\int_{\N}\frac{1}{|y|}\Ui^{\frac{2}{N-2}}\frac{\pa\Ui}{\pa\la_i}\Uj
dx-\va\int_{\N} \frac{K(x)}{|y|}\Ui^{\frac{N}{N-2}}\frac{\pa\Ui}{\pa\la_i}dx\\
&&+O\Bigl(\frac{\va\va_{12}}{\la_i}\Bigl)+O\Bigl(\frac{\va^{1+\tau}_{12}}{\la_i}\Bigl)
+O\Bigl(\frac{\va}{\la_i}\sum\limits_{j=1}^2
\Bigl(\frac{1}{\la_j^{\gamma_j+\sigma}}
+|\eta^j-\bar{\eta}^j|^{\gamma_j+\sigma}\Bigl)\Bigl)\\
&=&-C_{N,k}b_1\frac{\va_{12}}{\la_i}-\frac{(N-2)C_{N,k}}{2\la_i^{\gamma_i+1}}
\Bigl[\frac{b^i_2}{k}\sum_{j=1}^{k}\xi^i_j
+\frac{b^i_3}{h}\sum_{j=1}^{h}a^i_j\Bigl]\\
&&+O\Bigl(\frac{\va\va_{12}}{\la_i}\Bigl)+O\Bigl(\frac{\va^{1+\tau}_{12}}{\la_i}\Bigl)
+O\Bigl(\frac{\va}{\la_i}\sum\limits_{j=1}^2
\Bigl(\frac{1}{\la_j^{\gamma_j+\sigma}}
+|\eta^j-\bar{\eta}^j|^{\gamma_j+\sigma}\Bigl)\Bigl)+O\Bigl(\frac{\va}{\la_i^{\gamma_i}}|\eta_i-\bar{\eta}^i|\Bigl).\\
\end{eqnarray*}
\end{proof}
Similar to the proof of Lemma 4.1, using the estimates in Lemmas
5.5-5.7 in the Appendix  we obtain
\begin{lemma}\label{3.2}\hspace{0.5cm}  Under the same assumption as in Lemma 4.1, we have for $j,l=1,2,\,j\neq
l$, $i=1,\cdot\cdot\cdot,h$
\begin{eqnarray*}
\frac{\pa J(\eta,\la,v)}{\pa \eta^j_i}
 &=&-\frac{b_4^ja^j_i}{\la_j^{\gamma_j-2}}\va(\eta^j_i-\bar{\eta}^j_i)
 -C(\eta^j_i-\eta^l_i)\va_{12}+O\Bigl(\va\frac{\la^2_j|\eta^j-\bar{\eta}^j|^2}{\la_j^{\gamma_j-1}}\Bigl)
\\
 &&+O\Bigl(\va\sum\limits_{l=1}^2\Bigl(\frac{1}{\la_l^{\gamma_l-1+\sigma}}+\la_l|\eta^l-\bz^l|^{\gamma_l+\sigma}\Bigl)\Bigl)
 +O(\va\la_j\va_{12})+O(\la_j\va_{12}^{1+\tau}),
\end{eqnarray*}
where $C$ is a positive constant and $b_4^j$ is defined in Lemma
5.6.
\end{lemma}

\begin{lemma}\label{3.3}
For $(\eta,\la)\in D_\mu$ and $v(\eta,\la)$ obtained in
Proposition 3.2, $j=1,2,\,i=1,\cdot\cdot\cdot,h$,
\begin{eqnarray*}
B_j&=&O\Bigl(\la_j^2\Bigl(\sum\limits_{l=1}^2\Bigl(\frac{\va}{\la_l^{\gamma_l+1}}+
\va|\eta^l-\bar{\eta}^l|^{\gamma_l+1}\Bigl)\Bigl)\Bigl)+O(\la_j\va_{12}),\\
C_{ji}&=&O\Bigl(\frac{1}{\la_j^2}\Bigl(
\frac{\va}{\la_j^{\gamma_j-2}}|\eta^j-\bar{\eta}^j|+\va_{12}\Bigl)\Bigl)
+O\Bigl(\frac{\va}{\la_j}\sum\limits_{l=1}^2\Bigl(\frac{1}{\la_l^{\gamma_l+\sigma}}
+|\eta^l-\bar{\eta}^l|^{\gamma_l+\sigma}\Bigl)\Bigl).
\end{eqnarray*}
\end{lemma}
\begin{proof}
For each $\varphi\in \D$, there holds
$$
\Bigl\langle\frac{\pa J(\eta,\la,v)}{\pa
v},\,\varphi\Bigl\rangle=\sum\limits_{j=1}^2B_j
\Bigl\langle\frac{\pa \Uj}{\pa
  \la_j},\,\varphi\Bigl\rangle+\sum_{j=1}^2\sum\limits_{i=1}^{h} C_{ji}\Bigl\langle\frac{\pa \Uj}
  {\pa \eta_i^j},\,\varphi\Bigl\rangle.
$$
We take
$\varphi=\frac{\pa\Uj}{\pa\la_j},\,\,\frac{\pa\Uj}{\pa\eta^j_i}$,
$j=1,2,\,\,i=1,\cdot\cdot\cdot,h$ into the above equation, and use
the fact that
$$
\Bigl\langle \frac{\pa J}{\pa
v},\,\frac{\pa\Uj}{\pa\la_j}\Bigl\rangle=\frac{\pa
J}{\pa\la_j},\,\,\,\,\,\,\,\,\Bigl\langle \frac{\pa J}{\pa
v},\,\frac{\pa\Uj}{\pa\eta^j_i}\Bigl\rangle=\frac{\pa
J}{\pa\eta^j_i},
$$
then we  get a quasi-diagonal linear system with $B_j, \,C_{ji}$
as unknowns. Using the estimates in Lemmas 4.1, 4.2 and 5.8, the
required estimates can be obtained.
\end{proof}

\noindent{\bf Proof of Theorem 2.3.} Let
$$
L_\va=\va^{-\frac{\gamma_1\gamma_2}{\frac{N-2}{2}(\gamma_1
+\gamma_2)-\gamma_1\gamma_2}}.
$$
Now to complete the proof, we only need to show that \eqref{1.11},
\eqref{2.10} are satisfied by some $(\eta,\la)\in D_\mu$. We will
show that for some suitable $\delta>0$, $m_1>0$ small and $m_2$
large, there exists $(\eta,\la)\in D_\mu$ such that
$$
(\la_1,\la_2)\in
[m_1L_\va^{\gamma_1^{-1}},\,m_2L_\va^{\gamma_1^{-1}}]\times
[m_1L_\va^{\gamma_2^{-1}},\,m_2L_\va^{\gamma_2^{-1}}],\,\,\,\,\eta=(\eta^1,\eta^2)\in
B_{\frac{\delta}{\la_1}}(\bar{\eta}^1)\times
B_{\frac{\delta}{\la_2}}(\bar{\eta}^2)
$$
$(\eta,\la,v(\eta,\la))$ satisfies \eqref{1.11} and \eqref{2.10}.

Employing Lemmas 4.1-4.3, Lemma 5.8, we get the following
equivalent form of  \eqref{1.11}, \eqref{2.10}
\begin{eqnarray*}
\frac{\va}{\la_j^{\gamma_j}}(\eta^j_i-\bz^j_i)&=&O\Bigl(\va\sum\limits_{l=1}^2\Bigl(\frac{1}{\la_l^{\gamma_l+\sigma}}
+|\eta^l-\bar{\eta}^l|^{\gamma_l+\sigma}\Bigl)\Bigl)\\
&&+O(\frac{\va_{12}}{\la_j}),\,\,\,j=1,2,i=1,\cdot\cdot\cdot,h,\\
\frac{b_1}{(\la_1\la_2)^{\frac{N-2}{2}}}+\frac{N-2}{2}
\Bigl[\frac{b^j_2}{k}\sum_{l=1}^{k}\xi^j_l
+\frac{b^j_3}{h}\sum_{l=1}^{h}a^j_l\Bigl]\frac{\va}{\la_i^{\gamma_i}}&=&
O\Bigl(\va\sum\limits_{l=1}^2\Bigl(\frac{1}{\la_l^{\gamma_l+\sigma}}+|\eta^l-\bar{\eta}^l|^{\gamma_l+\sigma}\Bigl)\Bigl)\\
&&+O(\va\va_{12}+\va_{12}^{1+\tau}),\,\,j=1,2.
\end{eqnarray*}

Let
\begin{eqnarray*}
&\la_1=t_1L_\va^{\gamma_1^{-1}},\,\,\,\la_2=t_2L_\va^{\gamma_2^{-1}},\,\,\,t_j\in
[m_1,m_2],\\
&\eta^1-\bar{\eta}^1=\la_1^{-1}x^1,\,\,\,\eta^2-\bar{\eta}^2=\la_2^{-1}x^2,\,\,(x^1,x^2)\in
B_\delta(0)\times B_\delta(0)\subset\R^h\times\R^h.
\end{eqnarray*}
Then since $b_1<0$ and $\frac{b^j_2}{k}\sum_{l=1}^{k}\xi^j_l
+\frac{b^j_3}{h}\sum_{l=1}^{h}a^j_l>0$ by
 \eqref{g}, the above system can be rewritten in the following equivalent
one
\[
\left\{ \begin{array}{l}
 x^j=o_\va(1),\,\,\,j=1,2,\\
t_j^{-\gamma_j}-\frac{c_j}{(t_1t_2)^{\frac{N-2}{2}}}=o_\va(1),\,\,\,j=1,2,
\end{array}
\right.
\]
where $c_j=-b_1\Bigl(\frac{N-2}{2}
\Bigl[\frac{b^j_2}{k}\sum_{l=1}^{k}\xi^j_l
+\frac{b^j_3}{h}\sum_{l=1}^{h}a^j_l\Bigl]\Bigl)^{-1}>0$,
$o_\va(1)\to 0$ as $\va\to 0$.

Let
\begin{eqnarray*}
f(x^1,x^2)&=&(x^1,x^2),\,\,\,\,\,\,\,(x^1,x^2)\in
\Om_1:=B_\delta(0)\times
B_\delta(0),\\
g(t_1,t_2)&=&\Bigl(t_1^{-\gamma_1}-\frac{c_1}{(t_1t_2)^{\frac{N-2}{2}}},\,
t_2^{-\gamma_2}-\frac{c_2}{(t_1t_2)^{\frac{N-2}{2}}}\Bigl),\\
&&(t_1,\,t_2)\in \Om_2:=[m_1,\,m_2]\times [m_1,\,m_2].
\end{eqnarray*}

With the same arguments as in \cite{CNY3}, we can prove
$$
deg\bigl((f,g),\,\Om_1\times\Om_2,\,0\bigl)=-1.
$$
As a consequence, we complete the proof.\hspace{7cm} $\square$\\

\noindent{\bf Proof of Theorem 2.4.}\,\,\,The main idea  is from
\cite{Ya}. Since all the computations are similar to those in the
proof of Theorem 2.2, we only give the sketch.

For simplicity, we assume that
$(0,\bar{\eta}^{j_1})=(0,\bar{\eta}^1)$ and $(0,\bar{\eta}^2)$ is
another local maximum point of $\phi$ with $s\triangleq
|\bar{\eta}^1-\bar{\eta}^2|$ large enough. Define
$$
L_1=s^{\frac{(N-2)\gamma_2}{\gamma_1\gamma_2-(\gamma_1+\gamma_2)(N-2)/2}},\,\,\,
L_2=s^{\frac{(N-2)\gamma_1}{\gamma_1\gamma_2-(\gamma_1+\gamma_2)(N-2)/2}}.
$$

Proceeding as done in the proof of Proposition 3.2, we do the
finite dimensional reduction to obtain $v(\eta,\la)$ and the same
estimate on $v(\eta,\la)$.

We now study the problem
\begin{equation}\label{I}
\inf\{J(\eta,\la,v(\eta,\la)):\,\,\,(\eta,\la)\in D_{\mu,2}\},
\end{equation}
where
$$
D_{\mu,2}\triangleq\{(\eta,\la):\,\,(\eta,\la)\in
D_{\mu},\,\,\la_j\in [\be_1L_j,\,\be_2L_j],\,\,j=1,2\},
$$
$\be_1>0$ is a small constant and $\be_2>0$ is a large constant
and both of them will be determined later. Problem \eqref{I} has a
minimizer $(\tilde{\eta},\tilde{\la})\in D_{\mu,2}$. In the
sequel, We will prove that for $s$ large enough,
$(\tilde{\eta},\tilde{\la})$ is an interior point of $D_{\mu,2}$
and thus is a critical point of $J(\eta,\la,v(\eta,\la))$.

By Lemma 5.1, calculating  as in the proof of Proposition 3.2, we
obtain
\begin{eqnarray*}
J(\eta,\la,v(\eta,\la))&=&J(\eta,\la,0)+O(\|v\|^2))\\
&=&\sum\limits_{j=i}^2I\Bigl(\frac{U_{\eta^j,\la_j}}{\phi(0,\eta^j)^{(N-2)/2}}\Bigl)
-\frac{D\va_{12}}{\phi(0,\eta^1)^{(N-2)/2}\phi(0,\eta^2)^{(N-2)/2}}\\
&&+O\Bigl(\sum\limits_{j=1}^2\Bigl(|\eta^j-\bar{\eta}^j|^{2\gamma_j}
+\frac{1}{\la_j^{2\gamma_j}}\Bigl)\Bigl)+O(\va_{12}^{1 +2\tau})\\
&=&\sum\limits_{j=i}^2\Bigl[\Bigl(\frac{1}{2}\frac{1}{\phi(0,\eta^j)^{N-2}}-\frac{N-2}{2(N-1)}
\frac{\phi(0,\bar{\eta}^j)}{\phi(0,\eta^j)^{N-1}}\Bigl)A\\
&&+\frac{N-2}{2(N-1)}
\frac{1}{\phi(0,\eta^j)^{N-1}}\int_{\R^N}\frac{1}{|y|}Q_j\Bigl(\frac{x}{\la_j}
+(0,\eta^j)-(0,\bar{\eta}^j)\Bigl)U_{0,1}^{\frac{2(N-1)}{N-2}}\Bigl]\\
&&-\frac{D\va_{12}}{\phi(0,\eta^1)^{(N-2)/2}\phi(0,\eta^2)^{(N-2)/2}}\\
&&+O\Bigl(\sum\limits_{j=1}^2\Bigl(|\eta^j-\bar{\eta}^j|^{\gamma_j+\sigma}
+\frac{1}{\la_j^{\gamma_j+\sigma}}\Bigl)\Bigl),
\end{eqnarray*}
where $A=\int_{\N}\frac{1}{|y|}U_{0,1}^{\frac{2(N-1)}{N-2}}$ and
$D>0$ is a constant independent of $s$.

We first prove that
$|\tilde{\eta}^j-\bar{\eta}^j|<\frac{C}{\la_j}$.

Notice that
\begin{eqnarray*}
\frac{1}{2}\frac{1}{\phi(0,\eta^j)^{N-2}}-\frac{N-2}{2(N-1)}
\frac{\phi(0,\bar{\eta}^j)}{\phi(0,\eta^j)^{N-1}}&=&\frac{1}{2(N-1)}\frac{1}{\phi(0,\bar{\eta}^j)^{N-2}}\\
&&+O(|\eta^j-\bar{\eta}^j|^{2\gamma_j}),\\
Q_j\Bigl(\frac{x}{\la_j} +(0,\eta)-(0,\bar{\eta}^j)\Bigl)&\geq &
a_0|\eta^j-\bar{\eta}^j|^{\gamma_j}-C\frac{|x|^{\gamma_j}}{\la_{j}}.
\end{eqnarray*}
Hence from
$$
J(\tilde{\eta},\tilde{\la},v(\tilde{\eta},\tilde{\la}))\leq
J(\bar{\eta},\tilde{\la},v(\bar{\eta},\tilde{\la})),
$$
we conclude
$$
\sum\limits_{j=i}^2a_0|\tilde{\eta}^j-\bar{\eta}^j|^{\gamma_j}\leq
O\Bigl(\sum\limits_{j=1}^2\Bigl(|\tilde{\eta}^j-\bar{\eta}^j|^{2\gamma_j}
+\frac{1}{\la_j^{\gamma_j}}\Bigl)\Bigl)+O(\va_{12}^{1}),
$$
which implies $|\tilde{\eta}^j-\bar{\eta}^j|<\frac{C}{\la_j}$.

At last, we verify that $\tilde{\la}_j\in
(\be_1L_j,\be_2L_j),\,\,j=1,2.$

Denote $\tilde{\la}_j=t_jL_j,\,j=1,2$. By the fact $\gamma_j>N-2$,
we deduce that there exists $(t_{0,1},t_{0,2})\in\R^+\times\R^+$,
such that
\begin{equation}\label{r}
\sum\limits_{j=1}^2\frac{c}{t_{0,j}^{\gamma_j}}-\frac{D}{[t_{0,1}t_{0,2}
\phi(0,\bar{\eta}^1)\phi(0,\bar{\eta}^2)]^{(N-2)/2}}<-c'<0.
\end{equation}
Let $\la_0=(t_{0,1}L_1,\,t_{0,2}L_2)$. Then
\begin{equation}\label{s}
J(\bar{\eta},\la_0, v(\bar{\eta},\la_0))\leq
\sum\limits_{j=1}^2\frac{1}{2(N-1)}\frac{1}{\phi(0,\bar{\eta}^j)^{N-2}}A-c_0's^{-\frac{2\gamma_1\gamma_2(N-2)}
{2\gamma_1\gamma_2-(\gamma_1+\gamma_2)(N-2)}},
\end{equation}
for some constant $c_0'>0$. On the other hand, direct computation
gives
\begin{eqnarray*}
J(\tilde{\eta},\tilde{\la}, v(\tilde{\eta},\tilde{\la}))&\geq&
\sum\limits_{j=1}^2\frac{1}{2(N-1)}\frac{A}{\phi(0,\bar{\eta}^j)^{N-2}}\\
&&+c''\sum\limits_{j=1}^2\frac{1}{\tilde{\la}_j^{\gamma_j}}
-\frac{D\va_{12}}{\phi(0,\tilde{\eta}^1)^{(N-2)/2}\phi(0,\tilde{\eta}^2)^{(N-2)/2}}\\
&&+O\Bigl(\sum\limits_{j=1}^2\frac{1}{\tilde{\la}_j^{\gamma_j+\sigma}}+\va_{12}^{1
+2\tau}\Bigl)
\end{eqnarray*}
Hence, employing $
J(\tilde{\eta},\tilde{\la},v(\tilde{\eta},\tilde{\la}))\leq
J(\bar{\eta},\la_0,v(\bar{\eta},\la_0))$, we see
\begin{equation}\label{t}
c''\sum\limits_{j=1}^2\frac{1}{\tilde{\la}_j^{\gamma_j}}
-\frac{D\va_{12}}{\phi(0,\tilde{\eta}^1)^{(N-2)/2}\phi(0,\tilde{\eta}^2)^{(N-2)/2}}\leq
-c_0's^{-\frac{2\gamma_1\gamma_2(N-2)}
{2\gamma_1\gamma_2-(\gamma_1+\gamma_2)(N-2)}}.
\end{equation}

If $\tilde{\la}_1=\be_1L_1$, then
\begin{eqnarray*}
\va_{12}&=&\frac{1+o(1)}{(\tilde{\la}_1\tilde{\la}_2|\tilde{\eta}^1-\tilde{\eta}^2|^2)^{(N-2)/2}}
\leq\frac{2}{\be_1^{N-2}}s^{-\frac{2\gamma_1\gamma_2(N-2)}
{2\gamma_1\gamma_2-(\gamma_1+\gamma_2)(N-2)}}.
\end{eqnarray*}
Hence \eqref{t} implies
$$
\frac{c''}{\be_1^{\gamma_1}}-\frac{2}{\be_1^{N-2}}\leq -c_0'.
$$
This is impossible for $\be_1$ small since $\gamma_1>N-2$.

If $\tilde{\la}_1=\be_2L_1$, then
\begin{eqnarray*}
\va_{12}&=&\frac{1+o(1)}{(\tilde{\la}_1\tilde{\la}_2|\tilde{\eta}^1-\tilde{\eta}^2|^2)^{(N-2)/2}}
\leq\frac{\bar{c}}{\be_1^{(N-2)/2}\be_2^{(N-2)/2}}s^{-\frac{2\gamma_1\gamma_2(N-2)}
{2\gamma_1\gamma_2-(\gamma_1+\gamma_2)(N-2)}}.
\end{eqnarray*}
Hence \eqref{t} implies
$$
-\frac{\bar{c}}{\be_1^{(N-2)/2}\be_2^{(N-2)/2}}\leq -c_0'.
$$
This is impossible for $\be_2$ (depending on $\be_1$) sufficiently
large.  Since the same argument can be applied to $\tilde{\la}_2$,
we see $\tilde{\la}_j\in (\be_1L_j,\be_2L_j),\,\,j=1,2.$
\hspace{5cm}$\square$\\

Now we give the proof of Theorem 1.1 and Theorem 1.2.

\noindent{\bf Proof of Theorem 1.1 and Theorem
1.2:}\,\,\,\,\,Define
$$
D_{cy}^{1,2}(\N)\triangleq \{u\in\D:\,\,\,u(y,z)=u(|y|,z)\},
$$
then it is well known that the positive critical points of
$I|_{D_{cy}^{1,2}(\N)}$ are indeed critical points of $I$ in $\D$.
Hence, since Proposition 2.1 and Corollary 2.2, to complete the
proof, it suffices to prove that if $\phi(y,z)$ is cylindrically
symmetric in $y$, then so is $v_\va$  in Proposition 3.2.

   Indeed, if $\phi(y,z)=\phi(|y|,z)$, since $\frac{\pa\Uj}{\pa
\la_j},\,\,\frac{\pa\Uj}{\pa\eta^i_j}$ are cylindrically symmetric
in $y$, we can do in the proof of Proposition 3.2 the finite
dimensional reduction of $J$  to get the same results in
$D_{cy}^{1,2}(\N)$. \hspace{8.5cm} $\square$

\section{Appendix}

In this section we give some basic estimates  used in the previous
sections.

\begin{lemma}\hspace{0.3cm} Let $\al\geq\be>1$  such that
$\al+\be=\s$. Then
\begin{eqnarray*}
&&\int_{\R^{N}} \frac{1}{|y|}\Ui^{\frac{N}{N-2}}\Uj dx
=C_{N,k}\va_{ij}\int_{\R^{N}}\frac{dx}{|y|[(1+|y|)^2+|z|^2]^{\frac{N}{2}}}+(\va_{ij}^{1+\tau}),\\
&&\int_{\R^{N}}
\frac{1}{|y|}\Ui^{\al}\Uj^{\be}dx=O(\va_{ij}^{1+\tau}),
\end{eqnarray*}
where $\tau$ is some small positive constant.
\end{lemma}
\begin{proof}
We only prove the first estimate for the case $\la_1\geq\la_2$, all
the rest can be proved similarly. Split
\begin{eqnarray*}
\R^{N}&=&\{\bx: |x|\leq \sqrt{\la_1\la_2}/10\}\\
&&\bigcup\{|x|>
\sqrt{\la_1\la_2}/10:|x-\la_1(0,\eta^2-\eta^1)|\geq\la_1|\eta^2-\eta^1|/10\}\\
&&\bigcup\{|x|>
\sqrt{\la_1\la_2}/10:|x-\la_1(0,\eta^2-\eta^1)|<\la_1|\eta^2-\eta^1|/10\}\\
&=&:\Om_1\cup\Om_2\cup\Om_3.
\end{eqnarray*}
\begin{eqnarray*}
&&\int_{\R^{N}}\frac{1}{|y|} \U^{\frac{N}{N-2}}\V dx\\
&=&\int_{\R^{N}}
\frac{C_{N,k}}{|y|[(1+|y|)^2+|z|^2]^{\frac{N}{2}}}\frac{dx}{\Bigl(\frac{\la_1}{\la_2}(1+\frac{\la_2}{\la_1}|y|)^2
+|\sqrt{\frac{\la_2}{\la_1}}z-
\sqrt{\la_1\la_2}(\eta^2-\eta^1)|^2\Bigl)^{\frac{N-2}{2}}}\\
&=&C_{N,k}\va_{12}\int_{\Om_1}
\frac{dx}{|y|[(1+|y|)^2+|z|^2]^{\frac{N}{2}}}\Bigl(1
-\frac{(N-2)|y|}{\la_1/\la_2+\la_1\la_2|\eta^2-\eta^1|^2}\\
&&+
\frac{\la_2/\la_1O(|x|^2)}{\la_1/\la_2+\la_1\la_2|\eta^2-\eta^1|^2}\Bigl)+O(\va_{12})\int_{\Om_2}
\frac{C_{N,k} dx}{|y|[(1+|y|)^2+|z|^2]^{\frac{N}{2}}}\\
&&+O\Bigl(\frac{1}{\la_1^N|\eta^2-\eta^1|^N}\Bigl)\int_{\Om_3}\frac{dx}{|y|\Bigl(\Bigl(\frac{\la_1}{\la_2}(1+\frac{\la_2}{\la_1}|y|)^2
+|\sqrt{\frac{\la_2}{\la_1}}z-
\sqrt{\la_1\la_2}(\eta^2-\eta^1)|^2\Bigl)^{\frac{N-2}{2}}}\\
&=&C_{N,k}\va_{12}\int_{\R^{N}}
\frac{dx}{|y|[(1+|y|)^2+|z|^2]^{\frac{N}{2}}}
+O\Bigl(\va_{12}^{1+\tau}\Bigl)\\
&&+O\Bigl(\frac{1}{\la_1^N|\eta^2-\eta^1|^N}\Bigl)\int_0^{\la_2|\eta^2-\eta^1|/10}
\Bigl(\frac{\la_1}{\la_2}\Bigl)^{\frac{N}{2}}\frac{dx}{|y|\Bigl((1+|y|)^2+|z|^2\Bigl)^{\frac{N-2}{2}}}\\
&=&C_{N,k}\va_{ij}\int_{\R^{N}}\frac{dx}{|y|[(1+|y|)^2+|z|^2]^{\frac{N}{2}}}+(\va_{ij}^{1+\tau}).
\end{eqnarray*}
\end{proof}

\begin{lemma}\label{4.2} For $(\eta,\la)\in D_\mu$, $\mu$ small, we
have for $j\neq i,\,j,i=1,2,$
\begin{eqnarray*}
\frac{N}{N-2}\int_{\N}\frac{1}{|y|}\Uj^{\frac{2}{N-2}}\frac{\pa\Uj}{\pa\la_j}\Ui
dx
=b_1C_N\frac{\va_{12}}{\la_j}+O\Bigl(\frac{\va_{12}^{1+\tau}}{\la_j}\Bigl),
\end{eqnarray*}
where
\begin{eqnarray*}
b_1=\frac{N}{2}\int_{\R^{N}}\frac{(1-|x|^2)dx}{|y|\bigl[(1+|y|)^2+|z|^2\bigl]^{\frac{N+2}{2}}}.
\end{eqnarray*}
\end{lemma}
\begin{proof}
\begin{eqnarray*}
&&\frac{N}{N-2}\int_{\N}\frac{1}{|y|}\Uj^{\frac{2}{N-2}}\frac{\pa\Uj}{\pa\la_j}\Ui
dx\\
=&&\frac{N}{2\la_j}\int_{\N}\frac{1}{|y|}\Uj^{\frac{N}{N-2}}\Ui
dx-N\la_1\int_{\N}\frac{1}{|y|}\Uj^{\frac{N}{N-2}}\frac{(|y|+\frac{1}{\la_j})|y|+|z-\eta^j|^2}
{(1+\la_j|y|)^2+\la_j^2|z-\eta|^2}\Ui dx\\
=&&:I_1-NI_2.
\end{eqnarray*}
Proceeding as done  in the proof of Lemma 5.1, we find that
\begin{eqnarray*}
&&I_2\\
=&&\frac{C_{N,k}}{\la_j}\int_{\N}\frac{(1+|y|)|y|+|z|^2}{|y|[(1+|y|)^2+|z|^2]^{\frac{N+2}{2}}}
\frac{dx}{\Bigl(\frac{\la_1}{\la_2}(1+\frac{\la_2}{\la_1}|y|)^2
+|\sqrt{\frac{\la_2}{\la_1}}z-
\sqrt{\la_1\la_2}(\eta^2-\eta^1)|^2\Bigl)^{\frac{N-2}{2}}}\\
=&&
\frac{C_{N,k}\va_{12}}{\la_j}\int_{\N}\frac{(1+|y|)|y|+|z|^2}{|y|[(1+|y|)^2+|z|^2]^{\frac{N+2}{2}}}dx
+O\Bigl(\frac{\va_{12}^{1+\tau}}{\la_j}\Bigl).
\end{eqnarray*}

$I_1$ has been estimated in Lemma 5.1. As a result, we complete
the proof.
\end{proof}

\begin{lemma}\label{4.3} Under the assumptions of Lemma 5.2, we
have for $j=1,2,$
\begin{eqnarray*}
\int_{\N}\frac{K(x)}{|y|}\Uj^{\frac{N}{N-2}}\frac{\pa\Uj}{\pa\la_j}dx&=&\frac{(N-2)C_{N,k}}{2\la_j^{\gamma_j+1}}
\Bigl[\frac{b^j_2}{k}\sum_{i=1}^{k}\xi^j_i
+\frac{b^j_3}{h}\sum_{i=1}^{h}a^j_i\Bigl]+O\Bigl(\frac{1}{\la_j^{\gamma_j+1+\sigma}}\Bigl)\\
&&+
O\Bigl(\frac{1}{\la_j}|\eta^j-\bar{\eta}^j|^{\gamma_j+\sigma}\Bigl)
+O\Bigl(\frac{1}{\la^{\gamma_j}_j}|\eta^j-\bar{\eta}^j|\Bigl),
\end{eqnarray*}
where
\begin{eqnarray*}
b_2^j=\int_{\N}
\frac{|y|^{\gamma_j}(1-|y|^2-|z|^2)dx}{|y|[(1+|y|)^2+|z|^2]^{N}},\,\,\,\,\,b_3^j=\int_{\N}
\frac{|z|^{\gamma_j}(1-|y|^2-|z|^2)dx}{|y|[(1+|y|)^2+|z|^2]^{N}}.
\end{eqnarray*}
\end{lemma}
\begin{proof}
\begin{eqnarray*}
&&\int_{\N}\frac{K(x)}{|y|}\Uj^{\frac{N}{N-2}}\frac{\pa\Uj}{\pa\la_j}dx\\
&=&\int_{\{x\in\N:|x-(0,\bar{\eta}^j)|\leq
\delta\}}\Bigl(\sum\limits_{i=1}^{k}\xi^j_i|y_i^j|^{\gamma_j}+\sum\limits_{i=1}^{h}a^j_i|z_i-\bar{\eta}_i^j|^{\gamma_j}
+O(|x-(0,\bar{\eta}^j)|^{\gamma_j+\sigma})\Bigl)\times\\
&&\times\frac{1}{|y|}\Uj^{\frac{N}{N-2}}\frac{\pa\Uj}{\pa\la_j}dx+O\Bigl(\frac{1}{\la_j^{N}}\Bigl)\\
&=&\frac{C_{N,k}(N-2)}{2\la_j^{\gamma_j+1}}\int_{\N}\Bigl(\sum\limits_{i=1}^{k}\xi_i^j|y_i|^{\gamma_j}+\sum\limits_{i=1}^{h}
a^j_i
|z_i+\la_j(\eta_i^j-\bar{\eta}_i^j)|^{\gamma_j}\Bigl)\frac{(1-|y|^2-|z|^2)dx}{|y|[(1+|y|)^2+|z|^2]^{N}}\\
&&+O\Bigl(\frac{1}{\la_j^{\gamma_j+1+\sigma}}\Bigl)
+O\Bigl(\frac{1}{\la_j}|\eta^j-\bar{\eta}^j|^{\gamma_j+\sigma}\Bigl)\\
&=&\frac{C_{N,k}(N-2)}{2\la_j^{\gamma_j+1}}\int_{\N}\Bigl(\sum\limits_{i=1}^{k}\xi_i^j|y_i|^{\gamma_j}+\sum\limits_{i=1}^{h}
a^j_i
|z_i|^{\gamma_j}\Bigl)\frac{(1-|y|^2-|z|^2)dx}{|y|[(1+|y|)^2+|z|^2]^{N}}\\
&&+O\Bigl(\frac{1}{\la_j^{\gamma_j+1+\sigma}}\Bigl)+
O\Bigl(\frac{1}{\la_j}|\eta^j-\bar{\eta}^j|^{\gamma_j+\sigma}\Bigl)
+O\Bigl(\frac{1}{\la^{\gamma_j}_j}|\eta^j-\bar{\eta}^j|\Bigl)\\
&=&\frac{C_{N,k}(N-2)}{2k\la_j^{\gamma_j+1}}\sum\limits_{i=1}^{k}\xi_i^j
\int_{\N}\frac{|y|^{\gamma_j}(1-|y|^2-|z|^2)dx}{|y|[(1+|y|)^2+|z|^2]^{N}}\\
&&+\frac{C_{N,k}(N-2)}{2h\la_j^{\gamma_j+1}}\sum\limits_{i=1}^{h}a_i^j\int_{\N}
\frac{|z|^{\gamma_j}(1-|y|^2-|z|^2)dx}{|y|[(1+|y|)^2+|z|^2]^{N}}\\
&&+O\Bigl(\frac{1}{\la_j^{\gamma_j+1+\sigma}}\Bigl)+
O\Bigl(\frac{1}{\la_j}|\eta^j-\bar{\eta}^j|^{\gamma_j+\sigma}\Bigl)
+O\Bigl(\frac{1}{\la^{\gamma_j}_j}|\eta^j-\bar{\eta}^j|\Bigl)\\
 &=&\frac{(N-2)C_{N,k}}{2\la_j^{\gamma_j+1}}
\Bigl[\frac{b^j_2}{k}\sum_{i=1}^{k}\xi^j_i
+\frac{b^j_3}{h}\sum_{i=1}^{h}a^j_i\Bigl]\\
&&+O\Bigl(\frac{1}{\la_j^{\gamma_j+1+\sigma}}\Bigl)+
O\Bigl(\frac{1}{\la_j}|\eta^j-\bar{\eta}^j|^{\gamma_j+\sigma}\Bigl)
+O\Bigl(\frac{1}{\la^{\gamma_j}_j}|\eta^j-\bar{\eta}^j|\Bigl).
\end{eqnarray*}
\end{proof}

\begin{lemma}\label{4.4}
Suppose $(\eta,\la)\in D_\mu$, $v\in E_{\eta,\la}$. If $\mu$ and
$\va$ are small, then for $i=1,2,$
\begin{eqnarray*}
&&\Bigl|\int_{\N}\frac{1+\va K(x)}{|y|}\Bigl(\sum\limits_{j=1}^2\Uj\Bigl)^{\frac{2}{N-2}}\frac{\pa\Ui}{\pa\la_i}vdx\Bigl|\\
&=&O\Bigl(\frac{\va_{12}^{1/2+\tau}}{\la_i}+\frac{\va}{\la_i}\sum\limits_{j=1}^2
\Bigl(\frac{1}{\la_j^{\gamma_j}}+|\eta^j-\bar{\eta}^j|^{\gamma_j}\Bigl)\Bigl)\|v\|.
\end{eqnarray*}
\end{lemma}
\begin{proof} Similarly to the proof of Lemma 4.3, we have
\begin{eqnarray*}
&&\Bigl|\int_{\N}\frac{1+\va K(x)}{|y|}\Bigl(\sum\limits_{j=1}^2\Uj\Bigl)^{\frac{2}{N-2}}\frac{\pa\Ui}{\pa\la_i}vdx\Bigl|\\
&=&\Bigl|\int_{\N}\frac{1+\va
K(x)}{|y|}\Ui^{\frac{2}{N-2}}\frac{\pa\Ui}{\pa\la_i}vdx\Bigl|
+O\Bigl(\frac{\va_{12}^{1/2+\tau}}{\la_i}\Bigl)\|v\|\\
&=&(1+\va
K(0,\bar{\eta}^i))\Bigl\langle\frac{\pa\Ui}{\pa\la_i},\,v
\Bigl\rangle\\
&& +\va\Bigl|\int_{\pa\N}(
K(x)-K(0,\bar{\eta}^i))\frac{1}{|y|}\Ui^{\frac{2}{N-2}}\frac{\pa\Ui}{\pa\la_i}v dx\Bigl|+O\Bigl(\frac{\va_{12}^{1/2+\tau}}{\la_i}\Bigl)\|v\|\\
&=&O\Bigl(\frac{\va_{12}^{1/2+\tau}}{\la_i}+\frac{\va}{\la_i}\sum\limits_{j=1}^2
\Bigl(\frac{1}{\la_j^{\gamma_j}}+|\eta^j-\bar{\eta}^j|^{\gamma_j}\Bigl)\Bigl)\|v\|.
\end{eqnarray*}
\end{proof}

\begin{lemma}\label{4.5} For $(\eta,\la)\in D_\mu$, $\mu$ small, we
have for $j\neq i,\,j,i=1,2,\,l=1,\cdot\cdot\cdot,h,$
\begin{eqnarray*}
&&\frac{N}{N-2}\int_{\N}\frac{1}{|y|}\Uj^{\frac{2}{N-2}}\frac{\pa\Uj}{\pa\eta^j_l}\Ui
dx\\
=&&\frac{C_{N,k}(N-2)}{h}\va_{12}(\eta^j_l-\eta^i_l)
\int_{\R^{N}}\frac{|z|^2}{|y|[(1+|y|)^2+|z|^2]^{\frac{N+2}{2}}}dx\\
&&+O\Bigl(\la_j\va_{12}^{1+\tau}\Bigl).
\end{eqnarray*}
\end{lemma}
\begin{proof}
The proof can be completed with the same arguments as that of
estimate (F16) in \cite{Ba} and Lemma 4.1.
\end{proof}

\begin{lemma}\label{4.6} For $(\eta,\la)\in D_\mu$, $\mu$ small, we
have for $j=1,2,\,i=1,\cdot\cdot\cdot,h,$
\begin{eqnarray*}
\int_{\N}\frac{K(x)}{|y|}\Uj^{\frac{N}{N-2}}\frac{\pa\Uj}{\pa\eta^j_i}dx
&=&\frac{b_4^j(\eta^j_i-\bar{\eta}^j_i)a^j_i}{\la_j^{\gamma_j-2}}
+O\Bigl(\frac{\la_j^2|\eta^j-\bar{\eta}^j|^2}{\la_j^{\gamma_j-1}}\Bigl)\\
&&+O\Bigl(\frac{1}{\la_j^{\gamma_j-1+\sigma}}+\la_j|\eta^j-\bar{\eta}^j|^{\gamma_j+\sigma}\Bigl)
+O\Bigl(\frac{1}{\la_j^{N-1}}\Bigl).
\end{eqnarray*}
where
$$
b_4^j=\frac{C_{N,k}(N-2)\gamma_j}{h}\int_{\N}\frac{|z|^{\gamma_j}dx}{|y|[(1+|y|)^2+|z|^2]^N
}.
$$
\end{lemma}
\begin{proof}
\begin{eqnarray*}
&&\int_{\N}\frac{K(x)}{|y|}\Uj^{\frac{N}{N-2}}\frac{\pa\Uj}{\pa\eta^j_i}dx\\
&=&(N-2)\int_{\N}\frac{K(x)}{|y|}\Uj^{\s}\frac{\la_j^2(z_i-\eta^j_i)dx}{[(1+\la_j|y|)^2+\la_j^2|z-\eta^j|^2]
}\\
&=&(N-2)\int_{\{x\in\N:|x-(0,\bar{\eta}^j|\leq
\delta\}}\Bigl(\sum\limits_{l=1}^{k}\xi_l^j|y_i|^{\gamma_j}+\sum\limits_{l=1}^{h}a_l^j|z_i-\bar{\eta}^j_i|^{\gamma_j}
+ O(|x-(0,\bar{\eta}^j)|^{\gamma_j+\sigma})\Bigl)\\
&&\times\frac{1}{|y|}\Uj^{\s}\frac{\la_j^2(z_i-\eta^j_i)dx}{[(1+\la_j|y|)^2+\la_j^2|z-\eta^j|^2]
}
+O\Bigl(\frac{1}{\la_j^{N-1}}\Bigl)\\
&=&(N-2)\int_{\N}\Bigl(\sum\limits_{l=1}^{k}\xi_l^j|y_i|^{\gamma_j}
+\sum\limits_{l=1}^{h}a_l^j|z_i-\bar{\eta}^j_i|^{\gamma_j}
\Bigl)\frac{1}{|y|}\Uj^{\s}\frac{\la_j^2(z_i-\eta^j_i)dx}{[(1+\la_j|y|)^2+\la_j^2|z-\eta^j|^2]
}\\
&&O\Bigl(\frac{1}{\la_j^{\gamma_j-1+\sigma}}+\la_j|\eta^j-\bar{\eta}^j|^{\gamma_j+\sigma}\Bigl)
+O\Bigl(\frac{1}{\la_j^{N-1}}\Bigl)\\
&=&\frac{C_{N,k}(N-2)}{\la_j^{\gamma_j-1}}\int_{\N}\Bigl(\sum\limits_{l=1}^{k}\xi_l^j|y_i|^{\gamma_j}
+\sum\limits_{l=1}^{h}a_l^j|z_i+\la_j(\eta^j_i-\bar{\eta}^j_i)|^{\gamma_j}
\Bigl)\frac{z_idx}{|y|[(1+|y|)^2+|z|^2]^N
}\\
&&+O\Bigl(\frac{1}{\la_j^{\gamma_j-1+\sigma}}+\la_j|\eta^j-\bar{\eta}^j|^{\gamma_j+\sigma}\Bigl)
+O\Bigl(\frac{1}{\la_j^{N-1}}\Bigl)\\
&=&\frac{C_{N,k}(N-2)}{\la_j^{\gamma_j-1}}\int_{\N}
\sum\limits_{l=1}^{h}a_l^j\Bigl(|z_i|^{\gamma_j}+\gamma_j\la_j|z_i|^{\gamma_j-2}z_i(\eta^j_i-\bar{\eta}^j_i)
\Bigl)\frac{z_idx}{|y|[(1+|y|)^2+|z|^2]^N
}\\
&&+O\Bigl(\frac{\la_j^2|\eta^j-\bar{\eta}^j|^2}{\la_j^{\gamma_j-1}}\Bigl)+O\Bigl(\frac{1}{\la_j^{\gamma_j-1+\sigma}}+\la_j|\eta^j-\bar{\eta}^j|^{\gamma_j+\sigma}\Bigl)
+O\Bigl(\frac{1}{\la_j^{N-1}}\Bigl)\\
&=&\frac{C_{N,k}(N-2)\gamma_j}{h\la_j^{\gamma_j-2}}(\eta^j_i-\bar{\eta}^j_i)a^j_i\int_{\N}\frac{|z|^{\gamma_j}dx}{|y|[(1+|y|)^2+|z|^2]^N
}\\
&&+O\Bigl(\frac{\la_j^2|\eta^j-\bar{\eta}^j|^2}{\la_j^{\gamma_j-1}}\Bigl)+O\Bigl(\frac{1}{\la_j^{\gamma_j-1+\sigma}}+\la_j|\eta^j-\bar{\eta}^j|^{\gamma_j+\sigma}\Bigl)
+O\Bigl(\frac{1}{\la_j^{N-1}}\Bigl).
\end{eqnarray*}
\end{proof}
\begin{lemma}\label{4.7}
Suppose $(\eta,\la)\in D_\mu$, $v\in E_{\eta,\la}$. If $\mu$ and
$\va$ are small, then
\begin{eqnarray*}
&&\Bigl|\int_{\N}\frac{1+\va K(x)}{|y|}\Bigl(\sum\limits_{j=1}^2\Uj\Bigl)^{\frac{2}{N-2}}\frac{\pa\Uj}
{\pa\eta^j_i}v dx\Bigl|\\
&=&O\Bigl(\la_j\va_{12}^{1/2+\tau}+\la_j\va\sum\limits_{j=1}^2
\Bigl(\frac{1}{\la_j^{\gamma_j}}+|\eta^j-\bar{\eta}^j|^{\gamma_j}\Bigl)\|v\|.
\end{eqnarray*}
\end{lemma}
\begin{proof}
The proof is similar to Lemma 5.4.
\end{proof}
\begin{lemma}\label{4.8}
Suppose $(\eta,\la)\in D_\mu$, $v\in E_{\eta,\la}$. If $\mu$ and
$\va$ are small, then
\begin{eqnarray*}
&&\Bigl\langle\frac{\pa\Uj}{\pa\la_j},\,\frac{\pa\Ui}{\pa\la_i}\Bigl\rangle=\frac{A_1}{\la_i\la_j}\delta_{ij}
+\frac{C}{\la_i\la_j}\va_{ij}^{1+\tau},\\
&&\Bigl\langle\frac{\pa\Uj}{\pa\la_j},\,\frac{\pa\Ui}{\pa\eta^i_k}\Bigl\rangle=\left\{
\begin{array}{ll}
0,\,\,& i=j\vspace{0.2cm}\\
O(\la_i\va_{12}^{1+\tau}),&i\neq j,
\end{array}
\right.\\
&&\Bigl\langle\frac{\pa\Uj}{\pa\eta^j_l},\,\frac{\pa\Ui}{\pa\eta^i_k}\Bigl\rangle=\left\{
\begin{array}{ll}
A_2\la_i^2\delta_{lk}+O(\la^2_i\va_{12}^{1+\tau}),\,\,& i=j\vspace{0.2cm}\\
O(\la_i\la_j\va_{12}^{1+\tau}),&i\neq j,
\end{array}
\right.\\
 &&\Bigl\langle
v,\,\frac{\pa^2\Ui}{\pa\la_i^2}\Bigl\rangle=
O\Bigl(\ds\frac{\|v\|}{\la^2_i}\Bigl),\\
&&\Bigl\langle
v,\,\frac{\pa^2\Ui}{\pa\la_i\pa\eta^i_l}\Bigl\rangle= O(\|v\|),\\
&&\Bigl\langle
v,\,\frac{\pa^2\Ui}{\pa\eta^i_k\pa\eta^i_l}\Bigl\rangle=
O(\la_i^2\|v\|).
\end{eqnarray*}
\end{lemma}
\begin{proof}
The proof is similar to those of \cite{Ba} and Lemma 5.1.
\end{proof}
\begin{lemma}
Let $b_1$, $b_2^j$ and $b_3^j$ be defined as in Lemmas 5.2 and
5.3. Then,
\begin{eqnarray*}
b_1&=&-\frac{[k^2-2+k(h-1)+h]N\omega_h\omega_k}{2k(k+1)}\int_0^{+\infty}\frac{s^{k-2}ds}{(1+s)^{N-h}}
\int_0^{+\infty} \frac{t^{h-1}dt}{(1+t^2)^{\frac{N+2}{2}}}.\\
b_2^j&=&-\frac{(2N+2k-2)\omega_h\omega_k}{(2N-h-1)(2N-h-2)}\int_0^{+\infty}\frac{s^{\gamma_j+k-2}ds}{(1+s)^{2N-h-2}}
\int_0^{+\infty} \frac{t^{h-1}dt}{(1+t^2)^{N}}.\\
b_3^j&=&-\frac{2\gamma_j\omega_h\omega_k}{N-\gamma_j+k-2}\int_0^{+\infty}\frac{s^{k-2}ds}{(1+s)^{N-\gamma_j+k-2}}
\int_0^{+\infty} \frac{t^{\gamma_j+h-1}dt}{(1+t^2)^{N}}.
\end{eqnarray*}
Hence, $b_1<0,\,\,b_2^j<0$ and $b_3^j<0$.
\end{lemma}
\begin{proof}
We only prove  the estimate for $b_3^j$ since the estimates for
$b_1$ and $b_2^j$ are similar.

Firstly, it is easy to check that
\begin{eqnarray}
&&\int_0^{+\infty}
\frac{s^mds}{(1+s)^{n+1}}=\frac{n-m-1}{n}\int_0^{+\infty}
\frac{s^mds}{(1+s)^{n}},\,\,\,\,\forall \,0<m<n-1,\\
&&\int_0^{+\infty}
\frac{s^{m+1}ds}{(1+s)^{n+1}}=\frac{m+1}{n-m-1}\int_0^{+\infty}
\frac{s^mds}{(1+s)^{n}},\,\,\,\,\forall \,0<m<n-1,\\
&&\int_0^{+\infty}
\frac{t^{m-2}dt}{(1+t^2)^n}=\frac{2n-m-1}{2(n-1)}\int_0^{+\infty}
\frac{t^{m-2}dt}{(1+t^2)^{n-1}},\,\,\,\,\forall \,0<m<2n-1.
\end{eqnarray}
 Changing to polar coordinates
and using the change of variable $\bar{t}=\frac{t}{1+s}$, we can
find that
\begin{eqnarray*}
\frac{b_3^j}{\omega_h\omega_k}&=&\int_0^{+\infty}\int_0^{+\infty}\frac{s^{k-1}t^{\gamma_j+h-1}(1-s^2-t^2)}{s[(1+s)^2+t^2]^N}dsdt\\
&=&\int_0^{+\infty}\frac{s^{k-2}ds}{(1+s)^{2N-\gamma_j-h}}\int_0^{+\infty}\frac{t^{\gamma_j+h-1}dt}
{(1+t^2)^{N}}\\
&&-\int_0^{+\infty}\frac{s^{k}ds}{(1+s)^{2N-\gamma_j-h}}\int_0^{+\infty}
\frac{t^{\gamma_j+h-1}dt}{(1+t^2)^{N}}\\
&&-\int_0^{+\infty}\frac{s^{k-2}ds}{(1+s)^{2N-\gamma_j-h-2}}\int_0^{+\infty}\frac{t^{\gamma_j+h+1}dt}
{(1+t^2)^{N}}.
\end{eqnarray*}
Inserting (5.1)-(5.3) into the above equation, we see
\begin{eqnarray*}
\frac{b_3^j}{\omega_h\omega_k}&=&\Bigl[\int_0^{+\infty}\frac{s^{k-2}ds}{(1+s)^{2N-\gamma_j-h}}-
\int_0^{+\infty}\frac{s^{k}ds}{(1+s)^{2N-\gamma_j-h}}\\
&&-\frac{\gamma_j+h}{2N-\gamma_j-h-2}
\int_0^{+\infty}\frac{s^{k-2}ds}{(1+s)^{2N-\gamma_j-h-2}}\Bigl]\int_0^{+\infty}\frac{t^{\gamma_j+h-1}dt}{(1+t^2)^{N}}\\
&=&-\frac{2\gamma_j}{N-\gamma_j+k-2}\int_0^{+\infty}\frac{s^{k-2}ds}{(1+s)^{N-\gamma_j+k-2}}
\int_0^{+\infty} \frac{t^{\gamma_j+h-1}dt}{(1+t^2)^{N}}\\
 &<& 0.
\end{eqnarray*}
\end{proof}

\noindent{\bf Acknowledgements}:\,The first and the second author
were supported by the grant from NSFC(1063103). The first author
was also supported by Science Fund for Creative Research Groups of
NSFC(10721101) and CAS grant KJCX3- SYW-S03. The second author was
also supported by NCET(07-0350) and the Key Project of Chinese
Ministry of Education(107081). The third author was supported by
ARC of Australia.

\end{document}